\documentclass{article}%
\usepackage{amsfonts}
\usepackage{amsmath}
\usepackage{amssymb}
\usepackage{graphicx}
\usepackage{float}
\usepackage{comment}
\usepackage{setspace}
\usepackage[english]{babel}
\usepackage{xcolor}
\usepackage{caption}
\usepackage{subcaption}%
\providecommand{\U}[1]{\protect\rule{.1in}{.1in}}
\newtheorem{theorem}{Theorem}

\newtheorem{definition}[theorem]{Definition}

\newtheorem{lemma}[theorem]{Lemma}

\newtheorem{proposition}[theorem]{Proposition}
\newtheorem{remark}[theorem]{Remark}

\newenvironment{proof}[1][Proof]{\noindent\textbf{#1.} }{\ \rule{0.5em}{0.5em}}
\begin{document}

\title{On the generalized dimensions of physical measures of chaotic flows}
\author{Th\'{e}ophile Caby$%
{{}^\circ}%
$, Michele Gianfelice$^{\#}$\\$%
{{}^\circ}%
$CMUP, Departamento de Matem\'{a}tica \\
Faculdade de Ci\^{e}ncias, Universidade do Porto \\
Rua do Campo Alegre s/n \\
P-4169007 Porto\\
caby.theo@gmail.com\\
$^{\#}$Dipartimento di Matematica e Informatica\\
Universit\`{a} della Calabria\\
Campus di Arcavacata\\
Ponte P. Bucci - cubo 30B\\
I-87036 Arcavacata di Rende (CS)\\
gianfelice@mat.unical.it\\
\\
}
\date{}
\maketitle

\begin{abstract}
We prove that if $\mu$ is the physical measure of a $C^2$ flow in $\mathbb{R}^d, d \geq 3,$
diffeomorphically conjugated to a suspension flow based on a Poincar\'{e}
application $R$ with physical measure $\mu_{R}$, then $D_{q}(\mu)=D_{q}(\mu
_{R})+1$, where $D_{q}$ denotes the generalized dimension of order $q
\neq1$. We also show that a similar result holds for the local dimensions of $\mu$ and, 
under the additional hypothesis of exact-dimensionality of $\mu_{R}$, that our
result extends to the case $q=1$. We apply these results to estimate the
$D_{q}$ spectrum associated with R\"ossler systems and turn our attention to Lorenz-like flows,
proving the existence of their information dimension and giving a lower bound for their generalized dimensions.
\end{abstract}

\tableofcontents

\bigskip

\begin{description}
\item[AMS\ subject classification:] {\small 37C10, 37D45, 37E05}

\item[Keywords and phrases:] {\small generalized dimensions, chaotic
attractor, R\"ossler flow, singular-hyperbolic attractors.}
\end{description}

\bigskip

\section{Introduction}

Generalized dimensions have been introduced in 1983 by Hentschel and Procaccia
to describe the variety of different scaling behaviors of physical measures of
chaotic dynamical systems \cite{HP}. In addition, this spectrum of dimensions also
plays a central role in their statistical behavior.
In particular, they modulate their synchronization properties \cite{CFMVY,RT,BR}, 
as well the probability distributions of local quantities of interest such as local dimensions and recurrence times
\cite{CFMVY,CRS}. 

Let us now recall the definition of these objects and that of some related notions.

\begin{definition}
Let $\nu$ be a probability measure on $\left(  \mathbb{R}^{d},\mathcal{B}%
\left(  \mathbb{R}^{d}\right)  \right)  ,d\geq1$ and, for any $x\in
\mathbb{R}^{d},$ let $B_{r}^{\left(  d\right)  }\left(  x\right)  :=\left\{
y\in\mathbb{R}^{d}:\left\vert x-y\right\vert \leq r\right\}  $ be the
Euclidean ball of radius $r$ centered in $x.$

\begin{enumerate}
\item Setting
\begin{equation}
D_{q}^{+}\left(  \nu\right)  :=\left\{
\begin{array}
[c]{cc}%
\lim\sup_{r\downarrow0}\frac{\log\int\nu\left(  dx\right)
\nu^{q-1}\left(  B_{r}^{\left(  d\right)  }\left(  x\right)  \right)  }{(q-1)\log
r} & \text{if }q\neq1\\
\lim\sup_{r\downarrow0}\frac{\int\nu\left(  dx\right)  \log\nu\left(
B_{r}^{\left(  d\right)  }\left(  x\right)  \right)  }{\log r} & \text{if }q=1
\end{array}
\right.
\end{equation}
and
\begin{equation}
D_{q}^{-}\left(  \nu\right)  :=\left\{
\begin{array}
[c]{cc}%
\lim\inf_{r\downarrow0}\frac{\log\int\nu\left(  dx\right)
\nu^{q-1}\left(  B_{r}^{\left(  d\right)  }\left(  x\right)  \right)  }{(q-1)\log
r} & \text{if }q\neq1\\
\lim\inf_{r\downarrow0}\frac{\int\nu\left(  dx\right)  \log\nu\left(
B_{r}^{\left(  d\right)  }\left(  x\right)  \right)  }{\log r} & \text{if }q=1
\end{array}
\right.  \ ,
\end{equation}
we denote by $D_{q}\left(  \nu\right)  $ the generalized dimension of order
$q$ of $\nu$ provided these limits coincide.

\item Setting
\begin{equation}
d_{\nu}^{+}(x):=\limsup_{r\downarrow0}\frac{\log\nu(B_{r}^{(d)}(x))}{\log
r}\,
\end{equation}
and
\begin{equation}
d_{\nu}^{-}(x):=\liminf_{r\downarrow0}\frac{\log\nu(B_{r}^{(d)}(x))}{\log
r}\ ,
\end{equation}
we denote by $d_{\nu}\left(  x\right)  $ the local dimension at $x\in
\mathbb{R}^{d}$ of $\nu$ provided these limits coincide.

\item We say that $\nu$ is exact-dimensional if $d_{\nu}$ exists and is constant $\nu$-almost everywhere.
\end{enumerate}
\end{definition}

Generalized dimensions and local dimensions are closely related within the so-called multifractal formalism, 
which has been developed in analogy with thermodynamics \cite{Pe}. 
The $D_{q}$ spectrum provides a comprehensive description of the geometric
structure of the measure, including the information dimension $(q=1)$, the
correlation dimension $(q=2)$, and the box-counting dimension of the support
of the measure $(q=0)$. The values for $q=+\infty$ and $q=-\infty$ correspond
respectively to the infimum and the supremum values of the local dimensions. 
In different recent publications, it is used to gauge the
convergence of the geometric structure of the physical measures of coupled
chaotic systems, in a process that was termed \emph{topological synchronization}
\cite{L1,L2}.
Although many efforts have been put forward to numerically estimate these quantities, the designed algorithms 
are known to be subject to important numerical errors, especially for negative
$q$ \cite{PSR}. For flows, this adds up to the accumulated errors
associated with the discretization process, yielding unreliable estimates.

In this paper, we will study in detail the $D_{q}$ spectrum associated with
physical measures of chaotic flows in $\mathbb{R}^{d},d\geq 3,$ generated by
a $C^1$ vector field $V$ that are diffeomorphically conjugated to a
suspension flow. To be more specific, let $\Sigma $ be a compact $\left(
d-1\right) $-dimensional manifold that stands for a Poincar\'{e}\ surface
and $R:\Sigma \circlearrowleft $ a return map with physical measure $\mu
_{R}.$ We denote $\mathfrak{t}$ the associated roof function and suppose $%
\mathfrak{t}\in L^{1}(\mu _{R}).$ Setting $\mathfrak{s}_{0}:=0,\mathfrak{s}%
_{1}:=\mathfrak{t}$ and $\forall n\geq 2,$ 
\begin{equation}
\Sigma \ni \xi \longmapsto \mathfrak{s}_{n}\left( \xi \right)
:=\sum_{k=1}^{n}\mathfrak{t\circ }R^{k}\in \mathbb{R}_{+}\ ,
\end{equation}%
for any $\xi \in \Sigma $ we denote by 
\begin{align}
\Sigma \times \mathbb{R}_{+}\ni \left( \xi ,t\right) \mapsto \mathfrak{n}%
\left( \xi ,t\right) & :=\sum_{n\geq 0}\mathbf{1}_{\left[ 0,t\right] }\circ 
\mathfrak{s}_{n}\left( \xi \right) \label{counter} \\
& =\max \left\{ n\in \mathbb{Z}_{+}:\mathfrak{s}_{n}\left( \xi \right) \leq
t\right\} \in \mathbb{Z}_{+}  \notag
\end{align}%
and by 
\begin{equation}
\Sigma _{\mathfrak{t}}:=\left\{ \left( \xi ,t\right) \in \Sigma \times 
\mathbb{R}_{+}:t\in \lbrack 0,\mathfrak{t}\left( \xi \right) )\right\} \ .
\end{equation}%
We then define the semi-flow $\left( S^{t},\ t\geq 0\right) $ 
\begin{equation}
\Sigma _{\mathfrak{t}}\ni \left( \xi ,s\right) \longmapsto S^{t}\left( \xi
,s\right) :=\left( R^{\mathfrak{n}\left( \xi ,t+s\right) }\left( \xi \right)
,t+s-\mathfrak{s}_{\mathfrak{n}\left( \xi ,t+s\right) }\left( \xi \right)
\right) \in \Sigma _{\mathfrak{t}}\ ,\;t\geq 0\ .
\end{equation}%
whose physical measure $\mu _{S}$ has density $\frac{\mathbf{1}_{\left[ 0,%
\mathfrak{t}\right] }}{\mu _{R}\left[ \mathfrak{t}\right] }$ w.r.t. $\mu
_{R}\otimes \lambda ,$ where $\lambda $ is the Lebesgue measure on $\left( 
\mathbb{R},\mathcal{B}\left( \mathbb{R}\right) \right) .\left( S^{t},\ t\geq
0\right) $ naturally defines an equivalence relation among the elements of $%
\Sigma \times \mathbb{R}_{+}.$ Namely, $\left( \xi _{1},t_{1}\right) ,\left(
\xi _{2},t_{2}\right) \in \Sigma \times \mathbb{R}_{+}$ are said to be
equivalent, and this property is denoted by $\left( \xi _{1},t_{1}\right)
\sim \left( \xi _{2},t_{2}\right), $ if there exist $\left( \xi ,s\right) \in
\Sigma _{\mathfrak{t}},n_{1},n_{2}\in \mathbb{Z}_{+}$ such that $\left( \xi
_{1},t_{1}\right) =S^{\mathfrak{s}_{n_{1}}\left( \xi \right) }\left( \xi
,s\right)$ and $\left( \xi _{2},t_{2}\right) =S^{\mathfrak{s}_{n_{2}}\left( \xi
\right) }\left( \xi ,s\right) .$ Denoting by $\mathcal{V}:=\Sigma \times 
\mathbb{R}_{+}/\sim $ the quotient space and by $\pi :\Sigma \times \mathbb{R%
}_{+}\rightarrow \mathcal{V}$ the natural projection let us denote by $%
\mathcal{B}\left( \mathcal{V}\right) $ the Borel $\sigma $algebra generated
by the open sets of quotient topology which makes $\pi $ a continuous, hence
measurable, function. Then, denoting by $\left( \bar{S}^{t},\ t\geq 0\right) 
$ the semi-flow 
\begin{equation}
\Sigma \times \mathbb{R}_{+}\ni \left( \xi ,s\right) \longmapsto \bar{S}%
^{t}\left( \xi ,s\right) :=\left( R^{\mathfrak{n}\left( \xi ,t+s\right)
}\left( \xi \right) ,s+t\right) \in \Sigma \times \mathbb{R}_{+}\ ,\;t\geq
0\ ,
\end{equation}%
$\left( S^{t},\ t\geq 0\right) $ coincides with the semi-flow that is the
image under $\pi $ of $\left( \bar{S}^{t},\ t\geq 0\right) ,$ namely $\pi
\circ \bar{S}^{t}=S^{t}\circ \pi .$

We suppose that there exists a flow $\left( \chi ^{t},t\in \mathbb{R} \right) $
in $\mathbb{R}^{d},$ an open set $\mathcal{U}\subset \mathbb{R}^{d}$ such
that, setting 
\begin{gather}
\Sigma \times \mathbb{R}_{+}\ni \left( \xi ,t\right) \longmapsto \theta
\left( \xi ,t\right) :=\chi ^{t}\left( \xi \right) \in \mathcal{U\ }, \\
\theta \left( \xi ,\mathfrak{t}\left( \xi \right) \right) =\chi ^{\mathfrak{t%
}\left( \xi \right) }\left( \xi \right) =R\left( \xi \right) \ .
\end{gather}%
Hence, let $\Theta :\mathcal{V}\mapsto \mathcal{U}$ be the diffeomorphism
such, $\forall t\in \mathbb{R}_{+},$%
\begin{equation}
\Theta \circ S^{t}=\theta \circ \bar{S}^{t}\ \label{Theta}.
\end{equation}%
By construction $\pi $ acts on $\Sigma _{\mathfrak{t}}$ as the identity,
therefore $\pi \left( \mathring{\Sigma}_{\mathfrak{t}}\right) =\mathring{%
\Sigma}_{\mathfrak{t}}$ is open in $\mathcal{V}$ and consequently is in $%
\mathcal{B}\left( \mathcal{V}\right) .$ The physical measure $\mu $ of $%
\left( \chi ^{t},t\geq 0\right) $ is therefore the push-forward under $%
\Theta $ of $\mu _{S},$ namely $\mu :=\Theta _{\ast }\left( \mu _{S}\right)
, $ whose support is in $\Theta \circ \pi \left( \mathring{\Sigma}_{%
\mathfrak{t}}\right) =\Theta \left( \mathring{\Sigma}_{\mathfrak{t}}\right)
\subseteq \mathcal{U}.$\\
We also define 
\begin{align}
\Sigma _{\mathfrak{t}}& \ni \left( \xi ,t\right) \longmapsto \pi _{1}\left(
\xi ,t\right) :=t\in \mathbb{R}_{+}\ , \\
\Sigma _{\mathfrak{t}}& \ni \left( \xi ,t\right) \longmapsto \pi _{2}\left(
\xi ,t\right) :=\xi \in \Sigma \ .
\end{align}

We now aim to study the generalized dimensions associated with such systems, which encompass well-studied flows such as the classical Lorenz one.  
In \cite{RT}, it has been proven that for a suspension flow built on a
bi-Lipschitz base map $R$ and on a Lipschitz roof function $\mathfrak{t}$
the correlation dimension of the invariant measure for the flow $D_{2}\left(
\mu _{S}\right) $ satisfies the equality $D_{2}\left( \mu _{S}\right)
=D_{2}\left( \mu _{R}\right) +1,$ where $\mu _{R}$ is the invariant measure
for $R.$ Thus, in our framework, this implies $D_{2}\left( \mu \right)
=D_{2}\left( \mu _{R}\right) +1.$ In this paper we show that, for flows whose
physical measures can be constructed as presented above, this equality holds
not only for $q=2$ but for all $q\neq 1.$
We start by stating our main result, which relates $D_{q}(\mu)$ to $D_{q}(\mu_{R})$ for such flows. We then turn our attention to applications. We first discuss the case of the R\"ossler flow, since, from a numerical standpoint, the latter can be described as a suspension over a one-dimensional unimodal map, enabling to compute explicit estimates for its $D_{q}$ spectrum. In the final section, we focus on classical and geometric Lorenz flows, demonstrating the existence of their information dimension $D_{1}$ and providing a lower bound for their generalized dimensions.

\section{Main results}
The following theorem contains our major results relating the dimensions of a flow of the type just described to those of its base map.
\begin{theorem}
\label{mainT} Suppose $\mu$ is the physical measure of a $C^{2}$ flow in $\mathbb{R}^d$ constructed as in the previous section. Then, provided the different dimensions associated with $\mu_{R}$ exist, we have

\begin{enumerate}
\item For all $q\neq1$,
\begin{equation}
D_{q}\left(  \mu\right)  =D_{q}\left(  \mu_{R}\right)  +1\ . \label{Dq(mu)}%
\end{equation}

\item For all $x\in\mathcal{U}$,
\begin{equation}
d_{\mu}(x)=d_{\mu_{R}}(\xi)+1, \label{d(mu)}%
\end{equation}
where $\xi=\pi_{1}\circ\Theta^{-1}(x).$

\item If $\mu_{R}$ is exact-dimensional, then
\begin{equation}
D_{1}(\mu)=D_{1}(\mu_{R})+1. \label{D1(mu)}%
\end{equation}

\end{enumerate}
\end{theorem}
As it will appear from the proof, if the limits defining these dimensions do not exist, similar results hold for $D_q^{\pm}(\mu)$ and $d_\mu^{\pm}$.

\subsection*{Proof of point 1}
This result turns out to be a corollary of the case when the roof function
$\mathfrak{t}$ is bounded from above. Therefore we will first
introduce this case in the following proposition whose proof is deferred to the appendix.

\begin{proposition}    
\label{D_q}Assume that $\kappa \le \mathfrak{t} \le \kappa^{-1}$ for some $0<\kappa<1$. Then, for all $q\neq1$ such that $D_{q}\left(  \mu_{R}\right)$ exists,
we have
\begin{equation}
D_{q}\left(  \mu\right)  =D_{q}\left(  \mu_{R}\right)  +1\ . \label{D_q(mu)1}%
\end{equation}
\end{proposition}    

For the general case, first notice that, for any bounded measurable function $\varphi$ on $\Sigma$
and any $K>0,$%
\begin{align}
\mu_{R}\left[  \mathfrak{t}\varphi\right]   &  =\mu_{R}\left[  \mathfrak{t}%
\left(  \mathbf{1}_{\left[  0,K\right]  }\left(  \mathfrak{t}\right)
+\mathbf{1}_{\left(  K,\infty\right)  }\left(  \mathfrak{t}\right)  \right)
\varphi\right] \label{tK}\\
&  =\mu_{R}\left[  \left(  \mathfrak{t}\wedge K\right)  \varphi\right]
+\mu_{R}\left[  \left(  \mathfrak{t}-K\right)  \mathbf{1}_{\left(
K,\infty\right)  }\left(  \mathfrak{t}\right)  \varphi\right] \nonumber
\end{align}
and
\begin{align}
\left\vert \mu_{R}\left[  \left(  \mathfrak{t}-K\right)  \mathbf{1}_{\left(
K,\infty\right)  }\left(  \mathfrak{t}\right)  \varphi\right]  \right\vert  &
\leq\mu_{R}\left[  \left(  \mathfrak{t}-K\right)  \mathbf{1}_{\left(
K,\infty\right)  }\left(  \mathfrak{t}\right)  \left\vert \varphi\right\vert
\right] \\
&  \leq\left\Vert \varphi\right\Vert _{\infty}\int_{K}^{\infty}dt\mu
_{R}\left\{  \mathfrak{t}>t\right\}  \ .\nonumber
\end{align}
For any $\varepsilon>0,$ let $K_{\varepsilon}>0$ be such that $\int
_{K_{\varepsilon}}^{\infty}dt\mu_{R}\left\{  \mathfrak{t}>t\right\}
\leq\varepsilon. $ Then, $\forall K>K_{\varepsilon},\xi\in\Sigma,$%
\begin{align}
\mu_{R}\left[  \mathfrak{t}\mathbf{1}_{B_{r}^{\left(  d-1\right)  }\left(
\xi\right)  }\right]   &  \leq\mu_{R}\left[  \left(  \mathfrak{t}\wedge
K\right)  \mathbf{1}_{B_{r}^{\left(  d-1\right)  }\left(  \xi\right)
}\right]  +\varepsilon\\
&  \leq K\mu_{R}\left[  \mathbf{1}_{B_{r}^{\left(  d-1\right)  }\left(
\xi\right)  }\right]  +\varepsilon\ .\nonumber
\end{align}
Moreover, if $\varphi>0,$ from (\ref{tK}) we get that $\mu_{R}\left[
\mathfrak{t}\varphi\right]  \geq\mu_{R}\left[  \left(  \mathfrak{t}\wedge
K\right)  \varphi\right]  .$
Let us denote 
\begin{equation}
\nu_R(d\xi):=\frac{\mu_R(d\xi)\mathfrak{t}(\xi)}{\mu_R[\mathfrak{t}]}    
\end{equation}
and 
\begin{equation}
\nu_{R}^{K}\left(  d\xi\right)  :=\frac{\mu_{R}\left(  d\xi\right)  \left(
\mathfrak{t}\left(  \xi\right)  \wedge K\right)  }{\mu_{R}\left[
\mathfrak{t}\wedge K\right]  }\ .
\end{equation}
Point 1 follows immediately from Proposition \ref{D_q} and the following
Lemma.\newline

\begin{lemma}
\label{Ld5} For $q\neq 1$, $$D_{q}\left(  \nu_{R}\right)  =D_{q}\left(  \mu_{R}\right).$$
\end{lemma}

\begin{proof}\\

\noindent \underline{\bf Case $q>1$:} We have that
\begin{align}
\int_{\Sigma}\nu_{R}\left(  d\xi\right)  \nu_{R}^{q-1}\left(  B_{r}^{\left(
d-1\right)  }\left(  \xi\right)  \right)   &  =\frac{1}{\mu_{R}^{q}\left[
\emph{t}\right]  }\int_{\Sigma}\mu_{R}\left(  d\xi\right)  \mathfrak{t}\left(
\xi\right)  \mu_{R}^{q-1}\left[  \mathfrak{t}\mathbf{1}_{B_{r}^{\left(
d-1\right)  }\left(  \xi\right)  }\right] \label{lbq>1_i}\\
&  \geq\frac{\mu_{R}^{q}\left[  \mathfrak{t}\wedge K\right]  }{\mu_{R}%
^{q}\left[  \mathfrak{t}\right]  }\frac{\int_{\Sigma}\mu_{R}\left(
d\xi\right)  \left(  \mathfrak{t}\left(  \xi\right)  \wedge K\right)  \mu
_{R}^{q-1}\left[  \mathfrak{t}\mathbf{1}_{B_{r}^{\left(  d-1\right)  }\left(
\xi\right)  }\right]  }{\mu_{R}^{q}\left[  \mathfrak{t}\wedge K\right]
}\nonumber\\
&  \geq\frac{\mu_{R}^{q}\left[  \mathfrak{t}\wedge K\right]  }{\mu_{R}%
^{q}\left[  \mathfrak{t}\right]  }\frac{\int_{\Sigma}\mu_{R}\left(
d\xi\right)  \left(  \mathfrak{t}\left(  \xi\right)  \wedge K\right)  \mu
_{R}^{q-1}\left[  \left(  \mathfrak{t}\wedge K\right)  \mathbf{1}%
_{B_{r}^{\left(  d-1\right)  }\left(  \xi\right)  }\right]  }{\mu_{R}%
^{q}\left[  \mathfrak{t}\wedge K\right]  }\nonumber\\
&  =\frac{\mu_{R}^{q}\left[  \mathfrak{t}\wedge K\right]  }{\mu_{R}^{q}\left[
\mathfrak{t}\right]  }\int_{\Sigma}\nu_{R}^{K}\left(  d\xi\right)  \left(
\nu_{R}^{K}\left(  B_{r}^{\left(  d-1\right)  }\left(  \xi\right)  \right)
\right)  ^{q-1}\ .\nonumber
\end{align}
On the other hand, for any $\varepsilon>0,$%
\begin{equation}
\int_{\Sigma}\mu_{R}\left(  d\xi\right)  \mathfrak{t}\left(  \xi\right)
\nu_{R}^{q-1}\left[  \mathbf{1}_{B_{r}^{\left(  d-1\right)  }\left(
\xi\right)  }\right]  \leq\int_{\Sigma}\mu_{R}\left(  d\xi\right)  \left[
\left(  \mathfrak{t}\left(  \xi\right)  \wedge K\right)  \nu_{R}^{q-1}\left[
\mathbf{1}_{B_{r}^{\left(  d-1\right)  }\left(  \xi\right)  }\right]  \right]
+\varepsilon\ ,
\end{equation}
so that,
\begin{align}
\int_{\Sigma}\nu_{R}\left(  d\xi\right)  \nu_{R}^{q-1}\left(  B_{r}^{\left(
d-1\right)  }\left(  \xi\right)  \right)   &  \leq\frac{1}{\left(  \mu
_{R}\left[  \mathfrak{t}\wedge K\right]  \right)  ^{q}}\left\{  \int_{\Sigma
}\mu_{R}\left(  d\xi\right)  \left[  \left(  \mathfrak{t}\left(  \xi\right)
\wedge K\right)  \mu_{R}^{q-1}\left[  \mathfrak{t}\mathbf{1}_{B_{r}^{\left(
d-1\right)  }\left(  \xi\right)  }\right]  \right]  +\varepsilon\right\}
\label{ubq>1}\\
&  \leq\frac{\int_{\Sigma}\mu_{R}\left(  d\xi\right)  \left[  \left(
\mathfrak{t}\left(  \xi\right)  \wedge K\right)  \left(  \mu_{R}\left[
\left(  \mathfrak{t}\wedge K\right)  \mathbf{1}_{B_{r}^{\left(  d-1\right)
}\left(  \xi\right)  }\right]  +\varepsilon\right)  ^{q-1}\right]
+\varepsilon}{\left(  \mu_{R}\left[  \mathfrak{t}\wedge K\right]  \right)
^{q}}\ .\nonumber
\end{align}
Let $q\in(1,2].$ Since, for $a$ and $b>0,\left(  a+b\right)  ^{q-1}\leq
a^{q-1}+b^{q-1},$ we have
\begin{align}
\int_{\Sigma}\nu_{R}\left(  d\xi\right)  \nu_{R}^{q-1}\left(  B_{r}^{\left(
d-1\right)  }\left(  \xi\right)  \right)   &  \leq\int_{\Sigma}\nu_{R}%
^{K}\left(  d\xi\right)  \left(  \nu_{R}^{K}\left(  B_{r}^{\left(  d-1\right)
}\left(  \xi\right)  \right)  \right)  ^{q-1}+\label{ubq>1_1}\\
&  +\left(  \frac{\varepsilon}{\mu_{R}\left[  \mathfrak{t}\wedge K\right]
}\right)  ^{q-1}+\frac{\varepsilon}{\mu_{R}^{q}\left[  \mathfrak{t}\wedge
K\right]  }\nonumber\\
&  \leq\int_{\Sigma}\nu_{R}^{K}\left(  d\xi\right)  \left(  \nu_{R}^{K}\left(
B_{r}^{\left(  d-1\right)  }\left(  \xi\right)  \right)  \right)
^{q-1}+\nonumber\\
&  +\varepsilon^{q-1}\left(  \frac{K+\varepsilon^{2-q}}{\mu_{R}^{q}\left[
\mathfrak{t}\wedge K\right]  }\right)  \ .\nonumber
\end{align}
For $q>2,$ since
\begin{align}
\left(  \mu_{R}\left[  \left(  \mathfrak{t}\wedge K\right)  \mathbf{1}%
_{B_{r}^{\left(  d-1\right)  }\left(  \xi\right)  }\right]  +\varepsilon
\right)  ^{q-1}  &  =\mu_{R}^{q-1}\left[  \left(  \mathfrak{t}\wedge K\right)
\mathbf{1}_{B_{r}^{\left(  d-1\right)  }\left(  \xi\right)  }\right]  +\\
&  +\varepsilon\left(  q-1\right)  \int_{0}^{1}ds\left(  \mu_{R}\left[
\left(  \mathfrak{t}\wedge K\right)  \mathbf{1}_{B_{r}^{\left(  d-1\right)
}\left(  \xi\right)  }\right]  +s\varepsilon\right)  ^{q-2}\nonumber\\
&  \leq\mu_{R}^{q-1}\left[  \left(  \mathfrak{t}\wedge K\right)
\mathbf{1}_{B_{r}^{\left(  d-1\right)  }\left(  \xi\right)  }\right]
+\varepsilon\left(  q-1\right)  \left(  K+\varepsilon\right)  ^{q-2}%
\ ,\nonumber
\end{align}
from (\ref{ubq>1}) we get
\begin{align}
\int_{\Sigma}\nu_{R}\left(  d\xi\right)  \nu_{R}^{q-1}\left(  B_{r}^{\left(
d-1\right)  }\left(  \xi\right)  \right)   &  \leq\int_{\Sigma}\nu_{R}%
^{K}\left(  d\xi\right)  \left(  \nu_{R}^{K}\left(  B_{r}^{\left(  d-1\right)
}\left(  \xi\right)  \right)  \right)  ^{q-1}+\\
&  +\varepsilon\frac{\left(  q-1\right)  \mu_{R}\left[  \mathfrak{t}\wedge
K\right]  \left(  K+\varepsilon\right)  ^{q-2}+1}{\mu_{R}^{q}\left[
\mathfrak{t}\wedge K\right]  }\ .\nonumber
\end{align}
Then, setting
\begin{equation}
c_{q}\left(  K,\varepsilon\right)  :=\frac{\left(  q-1\right)  \mu_{R}\left[
\mathfrak{t}\wedge K\right]  \left(  K+\varepsilon\right)  ^{q-2}+1}{\mu
_{R}^{q}\left[  \mathfrak{t}\wedge K\right]  }\vee\frac{K+\varepsilon^{2-q}%
}{\mu_{R}^{q}\left[  \mathfrak{t}\wedge K\right]  }\ ,
\end{equation}
we have, for $r<1$
\begin{gather}
\frac{\log\int_{\Sigma}\nu_{R}\left(  d\xi\right)  \nu_{R}^{q-1}\left(
B_{r}^{\left(  d-1\right)  }\left(  \xi\right)  \right)  }{\log r} \leq
\frac{\log\int_{\Sigma}\nu_{R}^{K}\left(  d\xi\right)  \left(  \nu_{R}%
^{K}\left(  B_{r}^{\left(  d-1\right)  }\left(  \xi\right)  \right)  \right)
^{q-1}}{\log r}+\\
+\varepsilon^{\left(  q-1\right)  \wedge1}c_{q}\left(  K,\varepsilon\right)
\frac{\int_{0}^{1}ds\frac{1}{\int_{\Sigma}\nu_{R}^{K}\left(  d\xi\right)
\left(  \nu_{R}^{K}\left(  B_{r}^{\left(  d-1\right)  }\left(  \xi\right)
\right)  \right)  ^{q-1}+sc_{q}\left(  K,\varepsilon\right)  \varepsilon
^{\left(  q-1\right)  \wedge1}}}{\log r}\nonumber
\end{gather}
and by (\ref{lbq>1_i})
\begin{equation}
\frac{\log\int_{\Sigma}\nu_{R}\left(  d\xi\right)  \nu_{R}^{q-1}\left(
B_{r}^{\left(  d-1\right)  }\left(  \xi\right)  \right)  }{\log r}\geq
\frac{\log\int_{\Sigma}\nu_{R}^{K}\left(  d\xi\right)  \left(  \nu_{R}%
^{K}\left(  B_{r}^{\left(  d-1\right)  }\left(  \xi\right)  \right)  \right)
^{q-1}}{\log r}\ ,
\end{equation}
so that the result follows by taking the limit as $r\downarrow0$ in view of
Proposition \ref{D_q}.\\

\noindent \underline{\bf Case $q<1:$} For any $\varepsilon>0,$%
\begin{gather}
\int_{\Sigma}\nu_{R}\left(  d\xi\right)  \frac{1}{\nu_{R}^{1-q}\left(
B_{r}^{\left(  d-1\right)  }\left(  \xi\right)  \right)  }=\frac{1}{\mu
_{R}^{q}\left[  \mathfrak{t}\right]  }\int_{\Sigma}\mu_{R}\left(  d\xi\right)
\frac{\mathfrak{t}\left(  \xi\right)  }{\left(  \mu_{R}\left[  \mathfrak{t}%
\mathbf{1}_{B_{r}^{\left(  d-1\right)  }\left(  \xi\right)  }\right]  \right)
^{1-q}}\\
\geq\left(  \frac{\mu_{R}\left[  \mathfrak{t}\wedge K\right]  }{\mu_{R}\left[
\mathfrak{t}\right]  }\right)  ^{q}\frac{1}{\mu_{R}^{q}\left[  \mathfrak{t}%
\wedge K\right]  }\int_{\Sigma}\mu_{R}\left(  d\xi\right)  \frac
{\mathfrak{t}\left(  \xi\right)  }{\left(  \mu_{R}\left[  \left(
\mathfrak{t}\wedge K\right)  \mathbf{1}_{B_{r}^{\left(  d-1\right)  }\left(
\xi\right)  }\right]  +\varepsilon\right)  ^{1-q}}\nonumber\\
\geq\left(  \frac{\mu_{R}\left[  \mathfrak{t}\wedge K\right]  }{\mu_{R}\left[
\mathfrak{t}\right]  }\right)  ^{q}\frac{1}{\mu_{R}^{q}\left[  \mathfrak{t}%
\wedge K\right]  }\int_{\Sigma}\mu_{R}\left(  d\xi\right)  \frac
{\mathfrak{t}\left(  \xi\right)  \wedge K}{\left(  \mu_{R}\left[  \left(
\mathfrak{t}\wedge K\right)  \mathbf{1}_{B_{r}^{\left(  d-1\right)  }\left(
\xi\right)  }\right]  +\varepsilon\right)  ^{1-q}}\nonumber\\
\geq\left(  \frac{\mu_{R}\left[  \mathfrak{t}\wedge K\right]  }{\mu_{R}\left[
\mathfrak{t}\right]  }\right)  ^{q}\frac{1}{\mu_{R}^{q}\left[  \mathfrak{t}%
\wedge K\right]  }\int_{\Sigma}\mu_{R}\left(  d\xi\right)  \frac
{\mathfrak{t}\left(  \xi\right)  \wedge K}{\mu_{R}^{1-q}\left[  \left(
\mathfrak{t}\wedge K\right)  \mathbf{1}_{B_{r}^{\left(  d-1\right)  }\left(
\xi\right)  }\right]  }\times\nonumber\\
\times\left(  \frac{\mu_{R}\left[  \left(  \mathfrak{t}\wedge K\right)
\mathbf{1}_{B_{r}^{\left(  d-1\right)  }\left(  \xi\right)  }\right]  }%
{\mu_{R}\left[  \left(  \mathfrak{t}\wedge K\right)  \mathbf{1}_{B_{r}%
^{\left(  d-1\right)  }\left(  \xi\right)  }\right]  +\varepsilon}\right)
^{1-q}\nonumber\\
=\left(  \frac{\mu_{R}\left[  \mathfrak{t}\wedge K\right]  }{\mu_{R}\left[
\mathfrak{t}\right]  }\right)  ^{q}\int_{\Sigma}\nu_{R}^{K}\left(
d\xi\right)  \left(  \nu_{R}^{K}\left[  B_{r}^{\left(  d-1\right)  }\left(
\xi\right)  \right]  \right)  ^{q-1}\times\nonumber\\
\times\left(  1-\frac{\varepsilon}{\mu_{R}\left[  \left(  \mathfrak{t}\wedge
K\right)  \mathbf{1}_{B_{r}^{\left(  d-1\right)  }\left(  \xi\right)
}\right]  +\varepsilon}\right)  ^{1-q}\ .\nonumber
\end{gather}
But, since there exists $\kappa>0$ such that $\mathfrak{t}>\kappa,$
\begin{gather}
\left(  1-\frac{\varepsilon}{\mu_{R}\left[  \left(  \mathfrak{t}\wedge
K\right)  \mathbf{1}_{B_{r}^{\left(  d-1\right)  }\left(  \xi\right)
}\right]  +\varepsilon}\right)  ^{1-q}=\left(  1-\frac{\varepsilon}{\mu
_{R}\left[  \left(  \mathfrak{t}\wedge K\right)  \mathbf{1}_{B_{r}^{\left(
d-1\right)  }\left(  \xi\right)  }\right]  \nu_{R}^{K}\left[  \mathbf{1}%
_{B_{r}^{\left(  d-1\right)  }\left(  \xi\right)  }\right]  +\varepsilon
}\right)  ^{1-q}\\
\geq\left(  1-\frac{\varepsilon}{\kappa\nu_{R}^{K}\left[  \mathbf{1}%
_{B_{r}^{\left(  d-1\right)  }\left(  \xi\right)  }\right]  +\varepsilon
}\right)  ^{1-q}=1-\varepsilon\left(  1-q\right)  \int_{0}^{1}ds\frac{\left(
\kappa\nu_{R}^{K}\left[  \mathbf{1}_{B_{r}^{\left(  d-1\right)  }\left(
\xi\right)  }\right]  \right)  ^{1-q}}{\left(  \kappa\nu_{R}^{K}\left[
\mathbf{1}_{B_{r}^{\left(  d-1\right)  }\left(  \xi\right)  }\right]
+s\varepsilon\right)  ^{2-q}}\nonumber\\
\geq1-\frac{\varepsilon\left(  1-q\right)  }{\kappa\nu_{R}^{K}\left(
B_{r}^{\left(  d-1\right)  }\left(  \xi\right)  \right)  }\ .\nonumber
\end{gather}
Hence,
\begin{align}
\int_{\Sigma}\nu_{R}\left(  d\xi\right)  \frac{1}{\nu_{R}^{1-q}\left(
B_{r}^{\left(  d-1\right)  }\left(  \xi\right)  \right)  }  &  \geq\left(
\frac{\mu_{R}\left[  \mathfrak{t}\wedge K\right]  }{\mu_{R}\left[
\mathfrak{t}\right]  }\right)  ^{q}\int_{\Sigma}\nu_{R}^{K}\left(
d\xi\right)  \left(  \nu_{R}^{K}\left[  B_{r}^{\left(  d-1\right)  }\left(
\xi\right)  \right]  \right)  ^{q-1}\times\label{lbq<1}\\
&  \times\left(  1-\frac{\varepsilon\left(  1-q\right)  }{\kappa\nu_{R}%
^{K}\left(  B_{r}^{\left(  d-1\right)  }\left(  \xi\right)  \right)  }\right)
\nonumber\\
&  =\left(  \frac{\mu_{R}\left[  \mathfrak{t}\wedge K\right]  }{\mu_{R}\left[
\mathfrak{t}\right]  }\right)  ^{q}\left\{  \int_{\Sigma}\nu_{R}^{K}\left(
d\xi\right)  \left(  \nu_{R}^{K}\left[  B_{r}^{\left(  d-1\right)  }\left(
\xi\right)  \right]  \right)  ^{q-1}+\right. \nonumber\\
&  \left.  -\varepsilon\frac{\left(  1-q\right)  }{\kappa}\int_{\Sigma}\nu
_{R}^{K}\left(  d\xi\right)  \left(  \nu_{R}^{K}\left[  B_{r}^{\left(
d-1\right)  }\left(  \xi\right)  \right]  \right)  ^{q-2}\right\} \nonumber\\
&  \geq\left(  \frac{\mu_{R}\left[  \mathfrak{t}\wedge K\right]  }{\mu
_{R}\left[  \mathfrak{t}\right]  }\right)  ^{q}\int_{\Sigma}\nu_{R}^{K}\left(
d\xi\right)  \left(  \nu_{R}^{K}\left[  B_{r}^{\left(  d-1\right)  }\left(
\xi\right)  \right]  \right)  ^{q-1}\times\nonumber\\
&  \times\left[  1-\varepsilon\frac{\left(  1-q\right)  }{\kappa}\frac
{\int_{\Sigma}\nu_{R}^{K}\left(  d\xi\right)  \left(  \nu_{R}^{K}\left[
B_{r}^{\left(  d-1\right)  }\left(  \xi\right)  \right]  \right)  ^{q-2}}%
{\int_{\Sigma}\nu_{R}^{K}\left(  d\xi\right)  \left(  \nu_{R}^{K}\left[
B_{r}^{\left(  d-1\right)  }\left(  \xi\right)  \right]  \right)  ^{q-1}%
}\right]  \ ,\nonumber
\end{align}
so that, for $q<1$ and for any $\varepsilon>0,$%
\begin{gather}
\frac{\log\int_{\Sigma}\nu_{R}\left(  d\xi\right)  \nu_{R}^{q-1}\left(
B_{r}^{\left(  d-1\right)  }\left(  \xi\right)  \right)  }{\log r}\geq
\frac{\log\int_{\Sigma}\nu_{R}^{K}\left(  d\xi\right)  \left(  \nu_{R}%
^{K}\left(  B_{r}^{\left(  d-1\right)  }\left(  \xi\right)  \right)  \right)
^{q-1}}{\log r}+\\
+\frac{\log\left[  1-\varepsilon\frac{\left(  1-q\right)  }{\kappa}\frac
{\int_{\Sigma}\nu_{R}^{K}\left(  d\xi\right)  \left(  \nu_{R}^{K}\left[
B_{r}^{\left(  d-1\right)  }\left(  \xi\right)  \right]  \right)  ^{q-2}}%
{\int_{\Sigma}\nu_{R}^{K}\left(  d\xi\right)  \left(  \nu_{R}^{K}\left[
B_{r}^{\left(  d-1\right)  }\left(  \xi\right)  \right]  \right)  ^{q-1}%
}\right]  }{\log r}=\nonumber\\
\frac{\log\int_{\Sigma}\nu_{R}^{K}\left(  d\xi\right)  \left(  \nu_{R}%
^{K}\left(  B_{r}^{\left(  d-1\right)  }\left(  \xi\right)  \right)  \right)
^{q-1}}{\log r}+\nonumber\\
-\varepsilon\frac{\left(  1-q\right)  }{\kappa}\frac{\int_{0}^{1}ds\left[
\frac{\int_{\Sigma}\nu_{R}^{K}\left(  d\xi\right)  \left(  \nu_{R}^{K}\left[
B_{r}^{\left(  d-1\right)  }\left(  \xi\right)  \right]  \right)  ^{q-1}}%
{\int_{\Sigma}\nu_{R}^{K}\left(  d\xi\right)  \left(  \nu_{R}^{K}\left[
B_{r}^{\left(  d-1\right)  }\left(  \xi\right)  \right]  \right)  ^{q-2}%
}-s\varepsilon\frac{\left(  1-q\right)  }{\kappa}\right]  ^{-1}}{\log r}%
\geq\nonumber\\
\frac{\log\int_{\Sigma}\nu_{R}^{K}\left(  d\xi\right)  \left(  \nu_{R}%
^{K}\left(  B_{r}^{\left(  d-1\right)  }\left(  \xi\right)  \right)  \right)
^{q-1}}{\log r}+\nonumber\\
-\varepsilon\frac{\left(  1-q\right)  }{\kappa}\frac{\left[  \frac
{\int_{\Sigma}\nu_{R}^{K}\left(  d\xi\right)  \left(  \nu_{R}^{K}\left[
B_{r}^{\left(  d-1\right)  }\left(  \xi\right)  \right]  \right)  ^{q-1}}%
{\int_{\Sigma}\nu_{R}^{K}\left(  d\xi\right)  \left(  \nu_{R}^{K}\left[
B_{r}^{\left(  d-1\right)  }\left(  \xi\right)  \right]  \right)  ^{q-2}%
}-\varepsilon\frac{\left(  1-q\right)  }{\kappa}\right]  ^{-1}}{\log
r}\ .\nonumber
\end{gather}
Now, for any $q\neq1,$ there exists $C_{q}>0$ such that
\begin{equation}
C_{q}^{-1}r^{\left(  q-1\right)  D_{q}\left(  \mu_{R}\right)  }\leq
\int_{\Sigma}\nu_{R}^{K}\left(  d\xi\right)  \left(  \nu_{R}^{K}\left[
B_{r}^{\left(  d-1\right)  }\left(  \xi\right)  \right]  \right)  ^{q-1}\leq
C_{q}r^{\left(  q-1\right)  D_{q}\left(  \mu_{R}\right)  }\ .
\end{equation}
Therefore,
\begin{equation}
r^{\alpha_{q}}C_{q-1}^{-1}C_{q}^{-1}\leq\frac{\int_{\Sigma}\nu_{R}^{K}\left(
d\xi\right)  \left(  \nu_{R}^{K}\left[  B_{r}^{\left(  d-1\right)  }\left(
\xi\right)  \right]  \right)  ^{q-1}}{\int_{\Sigma}\nu_{R}^{K}\left(
d\xi\right)  \left(  \nu_{R}^{K}\left[  B_{r}^{\left(  d-1\right)  }\left(
\xi\right)  \right]  \right)  ^{q-2}}\leq C_{q-1}C_{q}r^{\alpha_{q}}\ ,
\end{equation}
where, since $q<1,$
\begin{align}
\alpha_{q}  &  :=\left(  q-1\right)  D_{q}\left(  \mu_{R}\right)  -\left(
q-2\right)  D_{q-1}\left(  \mu_{R}\right) \\
&  =\left(  2-q\right)  D_{q-1}\left(  \mu_{R}\right)  -\left(  1-q\right)
D_{q}\left(  \mu_{R}\right) \nonumber\\
&  =D_{q-1}\left(  \mu_{R}\right)  +\left(  1-q\right)  \left(  D_{q-1}\left(
\mu_{R}\right)  -D_{q}\left(  \mu_{R}\right)  \right)  >0\ .\nonumber
\end{align}
Taking the the limit $r\downarrow0,$ we get that, for
$q<1,D_{q}\left(  \nu_{R}\right)  \geq D_{q}\left(  \mu_{R}\right)  .$
Moreover,
\begin{align}
\int_{\Sigma}\nu_{R}\left(  d\xi\right)  \frac{1}{\nu_{R}^{1-q}\left(
B_{r}^{\left(  d-1\right)  }\left(  \xi\right)  \right)  }  &  =\left(
\frac{\mu_{R}\left[  \mathfrak{t}\wedge K\right]  }{\mu_{R}\left[
\mathfrak{t}\right]  }\right)  ^{q}\frac{1}{\mu_{R}^{q}\left[  \mathfrak{t}%
\wedge K\right]  }\int_{\Sigma}\mu_{R}\left(  d\xi\right)  \frac
{\mathfrak{t}\left(  \xi\right)  \wedge K}{\left(  \mu_{R}\left[  \left(
\mathfrak{t}\wedge K\right)  \mathbf{1}_{B_{r}^{\left(  d-1\right)  }\left(
\xi\right)  }\right]  \right)  ^{1-q}}\times\\
&  \times\left(  \frac{\mu_{R}\left[  \left(  \mathfrak{t}\wedge K\right)
\mathbf{1}_{B_{r}^{\left(  d-1\right)  }\left(  \xi\right)  }\right]  }%
{\mu_{R}\left[  \mathfrak{t}\mathbf{1}_{B_{r}^{\left(  d-1\right)  }\left(
\xi\right)  }\right]  }\right)  ^{1-q}\frac{\mathfrak{t}\left(  \xi\right)
}{\mathfrak{t}\left(  \xi\right)  \wedge K}\nonumber\\
&  \leq\left(  \frac{\mu_{R}\left[  \mathfrak{t}\wedge K\right]  }{\mu
_{R}\left[  \mathfrak{t}\right]  }\right)  ^{q}\int_{\Sigma}\nu_{R}^{K}\left(
d\xi\right)  \frac{1}{\left(  \nu_{R}^{K}\left(  B_{r}^{\left(  d-1\right)
}\left(  \xi\right)  \right)  \right)  ^{1-q}}\times\nonumber\\
&  \times\left(  1+\frac{\int_{\Sigma}\nu_{R}^{K}\left(  d\xi\right)  \left(
\nu_{R}^{K}\left(  B_{r}^{\left(  d-1\right)  }\left(  \xi\right)  \right)
\right)  ^{q-1}\frac{\mathfrak{t}\left(  \xi\right)  }{K}\mathbf{1}%
_{(K,\infty]}\left(  \mathfrak{t}\left(  \xi\right)  \right)  }{\int_{\Sigma
}\nu_{R}^{K}\left(  d\xi\right)  \left(  \nu_{R}^{K}\left(  B_{r}^{\left(
d-1\right)  }\left(  \xi\right)  \right)  \right)  ^{q-1}}\right)
\ .\nonumber
\end{align}
But, since there exists $\kappa>0$ such that $\mathfrak{t}>\kappa,$
\begin{align}
&  \frac{\int_{\Sigma}\nu_{R}^{K}\left(  d\xi\right)  \left(  \nu_{R}%
^{K}\left(  B_{r}^{\left(  d-1\right)  }\left(  \xi\right)  \right)  \right)
^{q-1}\frac{\mathfrak{t}\left(  \xi\right)  }{K}\mathbf{1}_{(K,\infty]}\left(
\mathfrak{t}\left(  \xi\right)  \right)  }{\int_{\Sigma}\nu_{R}^{K}\left(
d\xi\right)  \left(  \nu_{R}^{K}\left(  B_{r}^{\left(  d-1\right)  }\left(
\xi\right)  \right)  \right)  ^{q-1}}\\
&  \leq\left(  \frac{K}{\kappa}\right)  ^{q}\frac{\int_{\Sigma}\mu_{R}\left(
d\xi\right)  \left(  \mu_{R}\left(  B_{r}^{d-1}\left(  \xi\right)  \right)
\right)  ^{q-1}\frac{\mathfrak{t}\left(  \xi\right)  }{K}\mathbf{1}%
_{(K,\infty]}\left(  \mathfrak{t}\left(  \xi\right)  \right)  }{\int_{\Sigma
}\mu_{R}\left(  d\xi\right)  \mu_{R}\left(  B_{r}^{d-1}\left(  \xi\right)
\right)  ^{q-1}}\nonumber\\
&  =\frac{K^{q-1}}{\kappa^{q}}\frac{\int_{\Sigma}\mu_{R}\left(  d\xi\right)
\left(  \mu_{R}\left(  B_{r}^{d-1}\left(  \xi\right)  \right)  \right)
^{q-1}\mathfrak{t}\left(  \xi\right)  \mathbf{1}_{(K,\infty]}\left(
\mathfrak{t}\left(  \xi\right)  \right)  }{\int_{\Sigma}\mu_{R}\left(
d\xi\right)  \mu_{R}\left(  B_{r}^{d-1}\left(  \xi\right)  \right)  ^{q-1}%
}\nonumber
\end{align}
which implies
\begin{gather}
\int_{\Sigma}\nu_{R}\left(  d\xi\right)  \nu_{R}^{1-q}\left(  B_{r}^{\left(
d-1\right)  }\left(  \xi\right)  \right)  \leq\left(  \frac{\mu_{R}\left[
\mathfrak{t}\wedge K\right]  }{\mu_{R}\left[  \mathfrak{t}\right]  }\right)
^{q}\int_{\Sigma}\nu_{R}^{K}\left(  d\xi\right)  \frac{1}{\left(  \nu_{R}%
^{K}\left(  B_{r}^{\left(  d-1\right)  }\left(  \xi\right)  \right)  \right)
^{1-q}}\times\label{ubq<1_i}\\
\times\left(  1+\frac{K^{q-1}}{\kappa^{q}}\frac{\int_{\Sigma}\mu_{R}\left(
d\xi\right)  \left(  \mu_{R}\left(  B_{r}^{d-1}\left(  \xi\right)  \right)
\right)  ^{q-1}\mathfrak{t}\left(  \xi\right)  \mathbf{1}_{(K,\infty]}\left(
\mathfrak{t}\left(  \xi\right)  \right)  }{\int_{\Sigma}\mu_{R}\left(
d\xi\right)  \mu_{R}\left(  B_{r}^{d-1}\left(  \xi\right)  \right)  ^{q-1}%
}\right)  \ .\nonumber
\end{gather}
Hence, by (\ref{ubq<1_i}) we obtain
\begin{gather}
\frac{\log\int_{\Sigma}\nu_{R}\left(  d\xi\right)  \nu_{R}^{q-1}\left(
B_{r}^{\left(  d-1\right)  }\left(  \xi\right)  \right)  }{\log r}\leq
\frac{\log\int_{\Sigma}\nu_{R}^{K}\left(  d\xi\right)  \left(  \nu_{R}%
^{K}\left[  B_{r}^{\left(  d-1\right)  }\left(  \xi\right)  \right]  \right)
^{q-1}}{\log r}+\\
+q\frac{\log\frac{\mu_{R}\left[  \mathfrak{t}\wedge K\right]  }{\mu_{R}\left[
\mathfrak{t}\right]  }}{\log r}+\frac{\log\left(  1+\frac{K^{q-1}}{\kappa^{q}%
}\frac{\int_{\Sigma}\mu_{R}\left(  d\xi\right)  \left(  \mu_{R}\left(
B_{r}^{d-1}\left(  \xi\right)  \right)  \right)  ^{q-1}\mathfrak{t}\left(
\xi\right)  \mathbf{1}_{(K,\infty]}\left(  \mathfrak{t}\left(  \xi\right)
\right)  }{\int_{\Sigma}\mu_{R}\left(  d\xi\right)  \mu_{R}\left(  B_{r}%
^{d-1}\left(  \xi\right)  \right)  ^{q-1}}\right)  }{\log r}\nonumber
\end{gather}
which implies the thesis.
\end{proof}\\

\subsection*{Proof of point 2}

Proceeding as in Lemma \ref{B-C}, we can find a neighborhood $\mathcal{U}_{0}=%
\mathcal{U}_{0}^{+}\cup \mathcal{U}_{0}^{-}\subset \mathcal{U}$ of $\Sigma $
with $\mathcal{U}_{0}^{+}\cap \mathcal{U}_{0}^{-}=\Sigma ,$ such that, for
any $\left( \xi ,t\right) \in \Sigma _{\mathfrak{t}}$ with the property that 
$\Theta \left( \xi ,t\right) \in \mathcal{U}\backslash \mathcal{U}_{0},$ 
there exists $c>1$ such that, for $r$ sufficiently small,

\begin{equation}
\Theta (B_{c^{-1}r}^{(d-1)}(\xi )\times B_{c^{-1}r}^{(1)}(t)))\subset
B_{r}^{(d)}(\Theta \left( \xi ,t\right) )\subset \Theta (B_{cr}^{(d-1)}(\xi
)\times B_{cr}^{(1)}(t))
\end{equation}%
with 
\begin{equation}
B_{cr}^{(d-1)}(\xi )\times B_{cr}^{(1)}(t)\subset \mathring{\Sigma}_{%
\mathfrak{t}}\ .
\end{equation}%
Therefore,

\begin{equation}
\mu _{S}(B_{c^{-1}r}^{(d-1)}(\xi )\times B_{c^{-1}r}^{(1)}(t)))\leq \mu
(B_{r}^{(d)}\Theta (\xi ,t))\leq \mu _{S}(B_{cr}^{(d-1)}(\xi )\times
B_{cr}^{(1)}(t))\ .  \label{bc1}
\end{equation}%
On the other hand, if $\Theta \left( \xi ,t\right) \in \mathcal{U}_{0}^{+},$
for $s\in \left( \frac{\kappa }{4},\frac{\kappa }{2}\right) ,$ where $\kappa
>0$ is such that $\mathfrak{t}>\kappa ,$ there exists $c>1$ such that, by (%
\ref{B-C}), 
\begin{equation}
\mu _{S}\left( B_{c^{-1}r}^{\left( d-1\right) }\left( \xi \right) \times
B_{c^{-1}r}^{\left( 1\right) }\left( t+s\right) \right) \leq \mu \left( \chi
^{s}\left( B_{r}^{\left( d\right) }\left( \Theta \left( \xi ,t\right)
\right) \right) \right) \leq \mu _{S}\left( B_{cr}^{\left( d-1\right)
}\left( \xi \right) \times B_{cr}^{\left( 1\right) }\left( t+s\right)
\right) \ ,  \label{bc2}
\end{equation}%
with $B_{cr}^{\left( d-1\right) }\left( \xi \right) \times B_{cr}^{\left(
1\right) }\left( t+s\right) \subset \mathring{\Sigma}_{\mathfrak{t}},$ while
if $\Theta \left( \xi ,t\right) \in \mathcal{U}_{0}^{-},$%
\begin{equation}
\mu _{S}\left( B_{c^{-1}r}^{\left( d-1\right) }\left( \xi \right) \times
B_{c^{-1}r}^{\left( 1\right) }\left( t-s\right) \right) \leq \mu \left( \chi
^{-s}\left( B_{r}^{\left( d\right) }\left( \Theta \left( \xi ,t\right)
\right) \right) \right) \leq \mu _{S}\left( B_{cr}^{\left( d-1\right)
}\left( \xi \right) \times B_{cr}^{\left( 1\right) }\left( t-s\right)
\right) \ ,  \label{bc3}
\end{equation}%
with $B_{cr}^{\left( d-1\right) }\left( \xi \right) \times B_{cr}^{\left(
1\right) }\left( t-s\right) \subset \mathring{\Sigma}_{\mathfrak{t}}.$ Since 
$\mu $ is invariant for the flow $\left( \chi ^{t},t\in \mathbb{R}\right) $
and $\mu _{R}$ is preserved by $R,$ taking the $log$ of all the members of (%
\ref{bc1}), as well as of both sides of (\ref{bc2}) and (\ref{bc3}), and
dividing by $\log r,$ in the limit as $r\downarrow 0,$ we get 
\begin{equation}
d_{\mu }\left( \Theta \left( \xi ,t\right) \right) =d_{\mu _{R}}\left( \xi
\right) +1
\end{equation}%
and the result follows by setting $x:=\Theta \left( \xi ,t\right) .$

\subsection*{Proof of point 3}

As a consequence point 2, for $\mu$-almost all $x\in \mathcal{U} $, the function
\begin{equation}
\mathbb{R}^{d} \ni x \longmapsto \rho_r(x):=\frac{\log\mu(B_{r}^{(d)}(x))}{\log r} \in \mathbb{R}%
\end{equation}
converges, as $r\downarrow0,$ to $d_{\mu_{R}}(\pi_{1}\circ\Theta^{-1}(x))+1$, and $d_{\mu_{R}}(\xi)=D_{1}(\mu_{R})$ for $\mu_{R}$-almost every $\xi\in\Sigma$, since $\mu_{R}$ is exact-dimensional. Therefore, for $\mu$-almost every $x\in\mathcal{U}$, $|\rho_r(x)|$ is bounded. Thus, by the dominated convergence theorem, we get
\begin{align}
D_{1}(\mu)  &  =\lim_{r\downarrow0}\int   \rho_r(x) d\mu\left(  x\right)\\
&=\int  (\lim_{r\downarrow0}\rho_r(x)) d\mu\left(  x\right) \nonumber \\ 
&  =\int  (d_{\mu_{R}}(\pi_{2}\circ\Theta^{-1}(x))+1) d\mu\left(  x\right) \nonumber \\
&=D_{1}(\mu_{R})+1\ .\nonumber
\end{align}

\section{The R\"{o}ssler system}

Introduced by Otto R\"ossler in 1976 \cite{Ros}, this system is a well-known example of deterministic chaotic flow and has been widely studied in the physics literature on non-linear dynamics. The chaotic behavior of this system was proven in \cite{Zg}, although no proof has been put forward concerning the existence of a physical measure for this system. Yet, many publications have proposed numerical estimates for its generalized dimensions spectrum (see e.g. \cite{L2,SR}). In this section, we construct a suspension flow resembling the R\"ossler flow and compute explicitly its generalized dimension spectrum, based on our main Theorem. The equation of motion of the original R\"ossler flow is the following:
\begin{equation}
\mathbf{\dot{x}}=f_{c}\left(  \mathbf{x}\right) ,\label{LRoss}%
\end{equation}
where $\mathbf{x}=\left(  x,y,z\right)  \in\mathbb{R}^{3}$ and
\begin{equation}
f_{c}\left(  \mathbf{x}\right)  :=\left(
\begin{array}
[c]{c}%
-y-z\\
x+ay\\
b+z\left(  x-c\right)
\end{array}
\right)  \;,\;c>0. \label{Rfield}%
\end{equation}

Assuming that the parameters $a,b,c$ are chosen in such a way that the R\"{o}ssler's
system (\ref{LRoss}) exhibits a chaotic behaviour (see e.g. \cite{Zg}), the
associated flow can be described in terms of a suspension flow. Namely, following \cite{LDM}, we can shift the axes
origin to the unstable fixed point
\begin{equation}
p^-=(\varepsilon, -\varepsilon/2, \varepsilon/2),
\end{equation}
where $\varepsilon:=\frac{c-\sqrt{c^2-4ab}}2$, yielding the phase velocity field
\begin{equation}\label{field}
\bar{f}_c\left(  \mathbf{\bar{x}}\right):=\left\{
\begin{array}
[c]{l}%
-\bar{y}-\bar{z}\\
\bar{x}+a\bar{y}\\
\frac{\varepsilon}{a}\bar{x}+\bar{z}(  \bar{x}-c+\varepsilon)
\end{array},
\right.  
\end{equation}
where
\begin{align}
&  \left\{
\begin{array}
[c]{l}%
\bar{x}:=x-\varepsilon\\
\bar{y}:=y+\frac{\varepsilon}{a}\\
\bar{z}:=z-\frac{\varepsilon}{a}%
\end{array}
\right.  \ .
\end{align}

In cases when the R\"ossler attractor exhibits a chaotic behavior, $c$ is usually large with respect to $a$ and $b$ \cite{LDM}, so that $\varepsilon$ is very small. Therefore we can consider the velocity field

\begin{equation}\label{field2}
\tilde{f}_c\left(  \mathbf{\bar{x}}\right):=\left\{
\begin{array}
[c]{l}%
-\bar{y}-\bar{z}\\
\bar{x}+a\bar{y}\\
\bar{z}(\bar{x}-c)
\end{array}
\right.  
\end{equation}

\noindent as a small perturbation of $\bar{f}_c$, whose flow, as it clearly appears from Fig. \ref{roop}, also displays a chaotic
attractor $\mathcal{A}$ having similar features as the R\"{o}ssler's. Since, as it will become clear in an instant, $\tilde{f}_c$ is particularly convenient for an analytical treatment, we will from now on use this velocity field and assume that the system (\ref{field}) is sufficiently robust for the two attractors to have similar topological properties\footnote{To our
knowledge there are no rigorous results about the robustness of the R\"{o}ssler's
flow, although it appears from simulation \cite{Sp} that if the values of the
parameters are chosen away from the bifurcation points the R\"ossler system
is indeed robust.}. In the following, to ease the notation, we will set $\bar
{x}=x,\bar{y}=y$ and $\bar{z}=z$.
Let then $R:\Sigma\circlearrowleft$ be the first return map relative to the
cross-section
\begin{equation}
\Sigma:=\left\{  \left(  x,y,z\right)  \in\mathbb{R}^{3}:x=0,\dot
{x}>0\right\}  \ .
\end{equation}
We define the box:
\begin{equation}
B=\{(x,y,z)\in\mathbb{R}^{3},X\leq x\leq0,|y|\leq Y,|z|\leq Z\}\ ,
\end{equation}
where $X,Y,Z>0$ are such that
\begin{equation}
\{x\geq0\}\cap\mathcal{A}\subset B\ .
\end{equation}
See Fig. \ref{roop} for a graphical representation. Let $(x_{0},y_{0}%
,z_{0})\in B\cap\{x=0,y>0\}$. Whenever the flow is in $B$, since $x\leq0$,
using the Gr\"onwall lemma on the third component of $\tilde{f}_c$, we get
\begin{equation}
|z(t,x_{0},y_{0},z_{0}|)|\leq Ze^{-ct}\ ,
\end{equation}
for all $t$ such that $z(t,x_{0},y_{0},z_{0})\in B$. Let
\begin{equation}
t_{0}:=\inf\{t>0:\mathbf{x}(t_{0},0,y_{0},z_{0})\in\Sigma\}\ .
\end{equation}
We have that
\begin{equation}
t_{0}>\frac{\pi L}{\sup|E|}\geq\frac{\pi L}{\sqrt{(Y+Z)^{2}+(X+aY)^{2}%
+Z^{2}(X-c)^{2}}}\ ,
\end{equation}
\label{t0} 
where $L$ is the radius of the largest ball $B_{r}(0)$ such that
$B_{r}(0)\cap\mathcal{A}=\emptyset$ (its existence is insured by the
Hartman-Grobman Theorem). Numerical simulations performed for different values
of $c$ show that $L$ is slightly larger than $c/2$, while $X,Y<2c$ and $Z$ is
very small compared to $X$ and $Y$. This gives $t_{0}\leq1/2$, which implies
\begin{equation}
|z(t_{0},x_{0},y_{0},z_{0})|\leq Ze^{-c/2}\ .
\end{equation}
Note that, as suggested by Fig. \ref{roop}, $Z$ can be taken very small.
Indeed, the flow already undergoes a strong contraction for $x>c$. For the
parameters $c=18$, $a=b=0.1$, $Z$ can be taken of order $10^{-2}$. The
Poincar\'{e} map $R$ is then of the form
\begin{equation}
\Sigma\ni(y,z)\longmapsto R(y,z)=(R_{1}(y,z),R_{2}\left( y,z\right) )\in\Sigma\ .
\end{equation}
with $\left \vert R_{2}\left( y,z\right) \right \vert = O(e^{-c/2}).$ 
In the right-hand side of Fig. \ref{qqqq}, we plot the intersection of the
numerical attractor computed using the RK4 method at different discretization
steps $h$ with the Poincar\'{e} section $\Sigma$, using trajectories of sizes
$t=10^{4}/h$. Since the RK4 method has an accumulated error of order $h^{4}$,
the errors in the numerical computation of the trajectories of the flow
$\tilde{X}^t$ are of order $10^{4}h^{4}$, as
observed in Fig. \ref{qqqq}. Moreover, for all tested $h$, the intersection of
the numerically obtained attractor with $\Sigma$ appears to be embedded in a
one-dimensional curve that can be parameterized by $z=\varphi
_{h}(y)$. Consequently, the return
map $\hat{R}$ observed numerically is of the form
\begin{equation}
\Sigma\ni(y,z)\longmapsto\hat{R}(y,z)=(\mathbf{T}(y),\varphi_{h}%
(y))\in\Sigma\ ,
\end{equation}
with $|\varphi_{h}(y)|=O(h)$ and $\mathbf{T}$ is a unimodal map, as seen in Fig. 1 in \cite{LDM}, and the left-hand side of Fig.
\ref{qqqq}.
\begin{figure}[ptb]
\centering
\begin{minipage}[b]{1\textwidth}
\includegraphics[width=\textwidth]{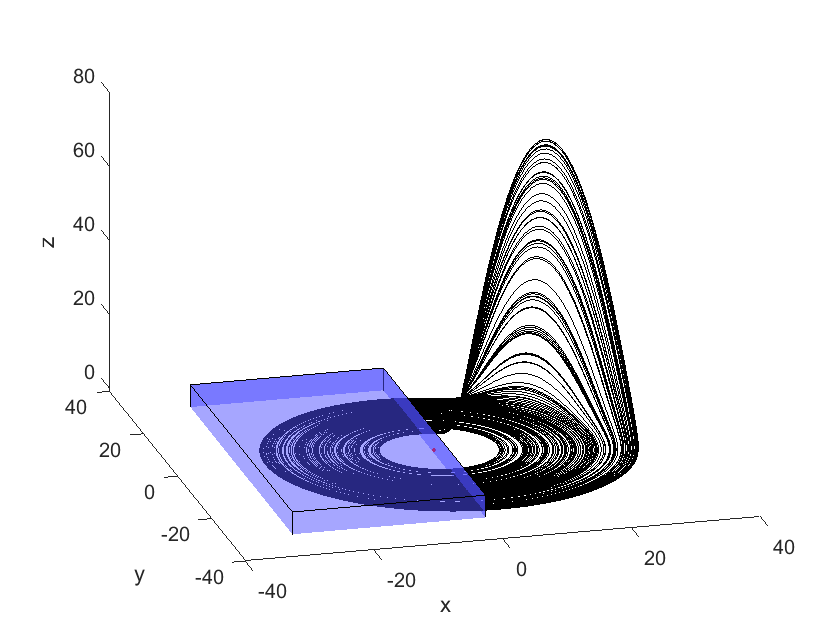}
\end{minipage}
\hfill\begin{minipage}[b]{1\textwidth}
\includegraphics[width=\textwidth]{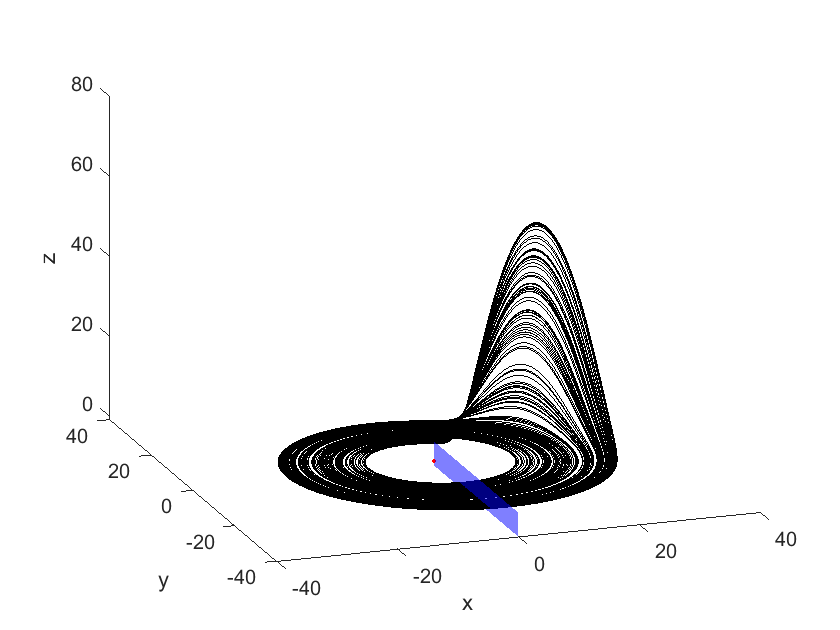}
\end{minipage}
\caption{Up: Attractor of the velocity field $\tilde{f}_c$. The box $B$ is represented
in blue. Down: Attractor of the flow generated by $\bar{f}_c$. The Poincar\'{e}
section $\Sigma$ is represented in blue. In both cases, we took with $a=b=0.1
$ and $c=18$}%
\label{roop}%
\end{figure}
\begin{figure}[ptb]
\centering
\begin{minipage}[b]{0.45\textwidth}
\includegraphics[width=\textwidth]{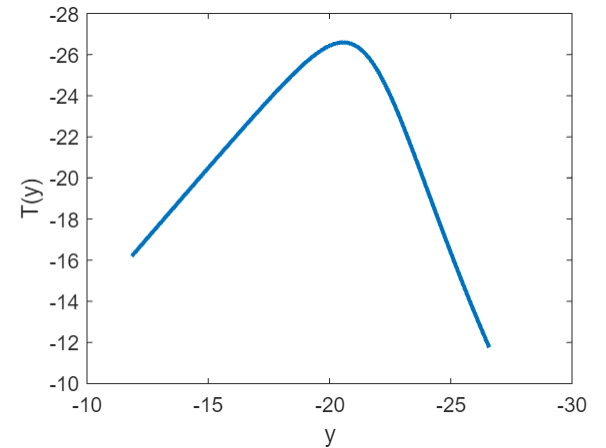}
\end{minipage}
\hfill\begin{minipage}[b]{0.45\textwidth}
\includegraphics[width=\textwidth]{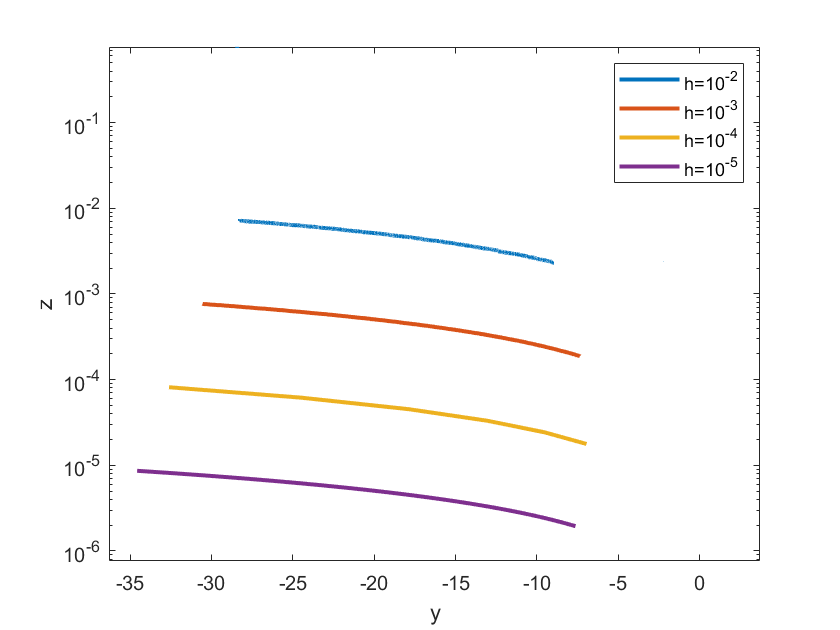}
\end{minipage}
\caption{Left: Graphical representation of the unimodal map $\mathbf{T}$,
associated with the R\"ossler flow of parameters $a=b=0.1$, $c=18$ (the axes
are inverted). Right: 1-D cross-section of the attractor with the Poincar\'{e}
section $\Sigma$ for different discretization steps $h$. }%
\label{qqqq}%
\end{figure}
This shows that the numerical estimate of the empirical average of an
observable $\phi$ of the R\"ossler's type semi-flow $\left(  \tilde{\chi}^{t},
t\geq0 \right)  $ defined by the vector field $\tilde{f}_c$, that is the empirical
average of $\phi$ computed along the trajectories of the numerically
integrated system, gets very close to the expected value $\hat{\mu}\left[
\phi\circ\Theta\right]  ,$ where $\Theta$ now stands for the diffeomorphism
conjugating $\left(  \tilde{\chi}^{t}, t\geq0 \right)  $ with its associated
suspension semi-flow and $\hat{\mu}$ is the measure with density
$\frac{\mathbf{1}_{\left[  0,\mathfrak{t}\right]  }}{\mu_{\hat{R}}\left[
\mathfrak{t}\right]  }$ w.r.t. $\mu_{\hat{R}}\otimes\lambda.$ In particular
this argument can be applied to estimate the generalized dimensions spectrum
of the physical measure $\mu$ of the flow defined by $\tilde{f}_c$ as it will be
highlighted in Remark \ref{Rem_Dq} at the end of this section. 
In order to do so, we start by proving Proposition \ref{DqRoss} below, which relies on the following 
lemma whose proof is deferred to the appendix.

\begin{lemma}
\label{lemm} Let $I$ and $J$ be two intervals such that $I\subseteq J$ and $%
T:I\circlearrowleft$ a map admitting a physical measure $\mu_{T}.$
Consider the map $\hat{R}:I\times J\circlearrowleft$ such that, given $%
\varphi :J\circlearrowleft$ a bounded $\mathcal{B}\left( J\right) $%
-measurable function, 
\begin{equation}
I\times J\ni\left( y,z\right) \longmapsto \hat{R}\left( y,z\right) :=\left(
T\left( y\right) ,\varphi\left( y\right) \right) \in I\times J\ .
\end{equation}
Then, for any $f\in BL\left( I\times J\right) ,$ where $BL\left( I\times
J\right) $ is the space of real-valued bounded Lipschitz functions on $%
I\times J$ of norm $1,$%
\begin{equation}
\lim_{n\rightarrow\infty}\frac{1}{n}\sum_{k=1}^{n}f\circ \hat{R}^{k}=\mu_{T}%
\left[ f\left( \cdot,\varphi\left( \cdot\right) \right) \right]
\;,\;dydz-a.s.\ .
\end{equation}
\end{lemma}

\begin{proposition} \label{DqRoss}
Denoting $\mu_\mathbf{T}$ the physical measure associated with $\mathbf{T}$, we have that $D_q({\mu}_{\hat{R}})=D_q(\mu_\mathbf{T})$ for all $q \in \mathbb{R}$, provided $%
D_q(\mu_\mathbf{T})$ exists.
\end{proposition}

\begin{proof}
The previous lemma implies that $\hat{\mu}_{R}$ is supported on $\Gamma _{h},$ the
graph of $\varphi _{h}$. Now,
\begin{equation}
\mu _{\Phi }:=\mu _{\mathbf{T}}\otimes \delta _{0}\
\end{equation}%
being the push-forward of $\hat{\mu}_{R}$ under the diffeomorphism 
\begin{equation}
\mathbb{R}^{2}\ni (y,z)\longmapsto \Phi (y,z)=(y,z-\varphi _{h}(y))\in 
\mathbb{R}^{2}\ ,
\end{equation}%
we have that 
\begin{equation}
D_{q}(\hat{\mu}_{R})=D_{q}(\mu _{\Phi })=D_{q}(\mu _{\mathbf{T}})
\end{equation}%
for all $q\in \mathbb{R}$ provided $D_{q}(\mu _{\mathbf{T}})$ exists.
\end{proof}\\
\begin{figure}[tbp]
\includegraphics[height=3in]{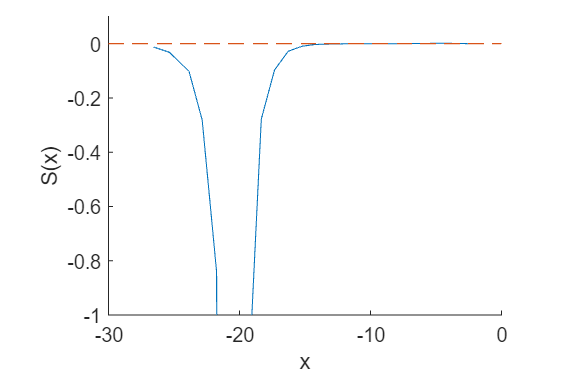}
\caption{Schwartzian derivative of $\mathbf{T}$ for the parameters $a=b=0.1$
and $c=18$.}
\label{sch}
\end{figure}

Let us now suppose that $\mathbf{T}$ is regular enough, so that its
Schwartzian derivative $\mathbf{S}$ is well-defined outside of the critical
point. This is a priori not implied by the regularity of the velocity field. Numerical
simulations suggest that $\mathbf{S}$ is negative, at least for the set of
parameters considered (See Figure 3). Under these assumptions, one can refer
to the well-developed theory of S-unimodal maps to compute the generalized
dimension spectrum of $\mu_\mathbf{T}$. Depending on the parameters $a$, $b$
and $c$, $\mu_{\mathbf{T}}$ is supported on either one of these 3 types of
sets \cite{Ke}:

\begin{enumerate}
\item a finite union of intervals, in which case $\mu_\mathbf{T}$ is
absolutely continuous with respect to Lebesgue,

\item a finite collection of points, in which case $\mu_\mathbf{T}$ is a sum of Dirac
measures,

\item a Cantor set.
\end{enumerate}
For our choice of parameters $a$, $b$ and $c$, we can reasonably rule out the second possibility, for otherwise, the dynamics of the system would be periodic.\\
Sets of type 1 and 3 can be hard to distinguish from numerical simulations, so, in order to discriminate between the two cases, we estimate numerically the Lyapunov exponent 
\begin{equation}
\lambda_{\mathbf{T}}(y)=\lim_{n\to\infty}\frac1n\sum_{i=1}^{n}\log|\mathbf{T}'(\mathbf{T}^i(y))| \, ,    
\end{equation}
for different generic points $y$ in the domain of $\mathbf{T}.$  The value of $\lambda_{\mathbf{T}}(y)$ turns out to be independent on the particular choice of $y$ and to be positive ($\approx 0.36$). From \cite{Ke}, $\mu_{\mathbf{T}}$ must therefore be supported on a set of type 1 rather than on a set of type 3 (which only occurs for a subset of parameters of zero Lebesgue measure in the set of parameters of $S-$unimodal maps \cite{Ke}). It is known that under such assumptions, the density is bounded away from 0 on its support and contains singularities distributed along the orbit of the critical point \cite{AM}. In particular, $\mu_\mathbf{T}$ should satisfy the hypothesis of the following Theorem, whose proof can be found in the appendix, which is enough to compute $D_q(\mu_\mathbf{T})$ explicitly.

\begin{theorem}\label{propp}
Let $I\subset \mathbb{R}$ be an interval and $\nu$ a probability measure on $\left( I,\mathcal{B}(I)\right)$ that is absolutely continuous with respect to the Lebesgue measure with density $\rho:I\to \mathbb{R}_{+}$. Suppose $\rho$ is bounded away from $0$ on its support and of the form

\begin{equation}
\rho=\psi_0 + \sum_{k=1}^{\infty} \frac{\psi_k\chi_k}{|x-x_k|^{\alpha_k}} \ ,
\end{equation}
where $\psi_0$ is bounded, $\{x_{k}\}_{k \geq 1}\subset I$ and, for any $k \geq 1,$

\begin{itemize}
\item $0<\alpha _{k+1} \leq\alpha _{k}<1;$

\item For all $k\ge 1$, $\psi_k$ is continuous at $x_k$ and $\psi_{k}(x_k)\neq 0;$

\item $\chi _{k}$ is either $\mathbf{1}_{[x_{k},+\infty )}$ or $\mathbf{1}%
_{(-\infty ,x_{k}]}.$
\end{itemize}
Then denoting $\alpha=\underset{k\ge 1}\sup \{\alpha_k\}=\alpha_1,$ we have

\begin{equation}
D_{q}(\nu)=
\begin{cases}
1 \ \text{if} \  q<1/\alpha,\\
\frac{q(1-\alpha)}{q-1} \ \text{otherwise}. 
\end{cases}
\end{equation}

\end{theorem}

\begin{remark}
For certain quadratic maps of Benedicks-Carleson type, $\alpha_k=1/2$ for all $k$ \cite{BS}. For such maps, we computed in \cite{CGSV} a quantity obtained by taking the Legendre transform of the singularity spectrum $f(\alpha)$ of the measure $\nu$. This quantity is known to coincide with $ D_q(\nu)$ under some conditions on $f$ which are not verified in the present context (in particular, $f$ must be defined on an interval \cite{Pe}). Remarkably, our formula agrees with the quantity computed in \cite{CGSV}, even though we are in the presence of countably infinitely many singularities, suggesting that this computation method remains valid in a broader context. 
\end{remark}
Applying Theorem \ref{mainT}, we get the following formula:

\begin{equation}\label{dqrr}
D_{q}(\hat{\mu})=
\begin{cases}
2 \ \text{if} \  q< 1/\alpha,\\
1+ \frac{q(1-\alpha)}{q-1} \ \text{otherwise} \ . 
\end{cases}
\end{equation} 
Several publications give estimates of the correlation dimension $D_2$ for the R\"ossler system that are compatible with Formula (\ref{dqrr}). For example, \cite{SR} gives $D_2=1.986 \pm 0.078$
. Although these estimates are obtained for the parameters $a=0.1,b=0.1,c=14$, Formula (\ref{dqrr}) still holds if $\mathbf{T}$ is a regular $S-$unimodal map and $\mu_{\mathbf{T}}$ is absolutely continuous. We believe that it is a generic situation for choices of $a,b,c$ such that the R\"ossler system is chaotic, since the set of unimodal maps having a Cantor set as an attractor has zero Lebesgue measure in the set of parametrizable S-unimodal maps. The values of these parameters may influence the value of $\alpha$, which depends on the behavior of the unimodal map near its critical point. Estimates of $D_q$ for negative $q$ are known to be subject to important statistical errors \cite{PSR}, which could explain the discrepancy between the curve in Figure 1 in \cite{L2} and Formula (\ref{dqrr}) in this range.\\

\begin{figure}[!tbp]
  \centering
    \includegraphics[width=8cm]{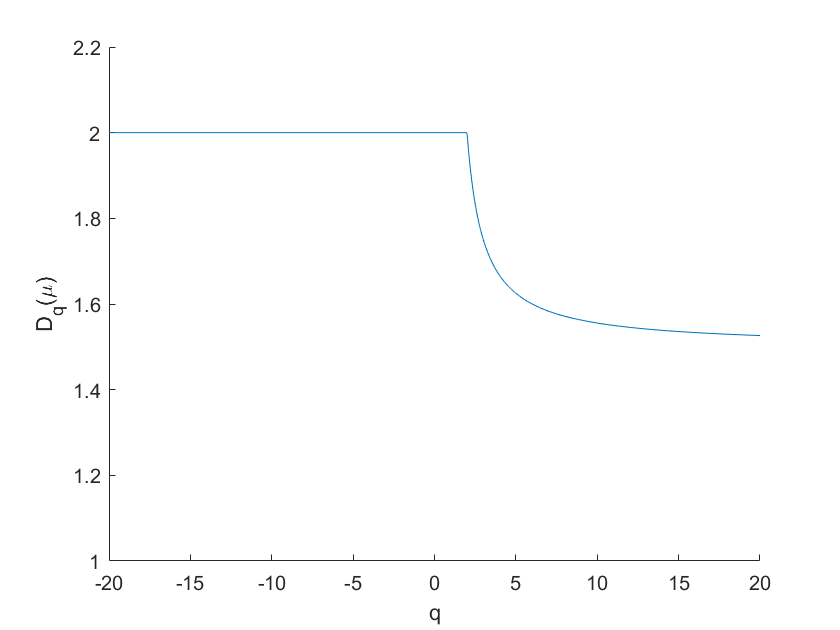}
  \caption{Representation of the $D_q$ spectrum of the measure $\hat{\mu}$, as given by formula (\ref{dqrr}) for $\alpha=1/2$. }
  \label{dqqqq}
  \end{figure}

\begin{remark} \label{Rem_Dq}
$\left . \right .$ \linebreak
\begin{enumerate} 
    \item Although we considered $R$ to be unidimensional, it is in fact a two-dimensional map and the support of $\mu_R$ should exhibit a fractal structure in the $z$ direction, although invisible in numerical experiments. The contribution of this vertical component should overall increase the values of its $D_q$ spectrum.

    \item In order to justify analytically that the strong contraction of the flow in the vertical direction yields a Poincar\'e application that has the features of a unidimensional map, our analysis has been carried on for the velocity field $\tilde{f}_c$, but the same considerations apply to the original R\"ossler field $f_c$. In particular, it was observed in several numerical studies the same features hold for the original R\"ossler flow's Poincar\'e map (see for instance \cite{LDM}).

    \item For values of $a,b,c$ such that the R\"ossler system evolves toward a periodic orbit, the invariant measure associated with the return map is a sum of Dirac measures, so that $D_q(\mu_R)=0$ for all $q$. In that case, Theorem \ref{mainT} yields $D_q(\mu)=1$ for all $q$, as expected for a measure supported on a closed curve.
\end{enumerate}
\end{remark}

\section{Lorenz-like flows}

Conversely to the R\"{o}ssler flow, singular-hyperbolic attractors (a.k.a.
Lorenz-like flows), are known to admit a unique SRB measure which is also
the physical measure of the system \cite{APPV}. In particular this holds for
the so-called geometric Lorenz flows, which encapture the main topological
features of singular hyperbolic attractors (\cite{APPV} Corollary 3).
Moreover, by \cite{MPP} and \cite{Tu} the classical Lorenz attractor is
singular and in view of \cite{AMe}, \cite{HM} Proposition 2.6 and \cite{Tu}
Section 2.4, the same considerations on the physical measure of Lorenz-like
flows apply to the classical Lorenz flow. In particular, Theorem C in \cite%
{APPV} holds also in this case as stated in \cite{APPV} Corollary 4.

The Haussdorff dimension of geometric Lorenz attractors is
known to be strictly greater than $2$ \cite{PRL}, while in \cite{GP} it has
been proved that the physical measure $\mu _{R}$ of the Poincar\'{e} map $R$
on which that of the return map $R$ corresponding Lorenz-like flow $\mu $ is constructed,\
following the scheme presented in the introduction, is exact dimensional. To
our knowledge no other result concerning the spectrum of generalized
dimensions of the physical measure of singular-hyperbolic attractors appear
in the literature. Here we make a step forward in this direction although a
complete answer to problem of computing $D_{q}\left( \mu \right) ,q\in 
\mathbb{R},$ is still out of reach.

In order to make this section self-contained further notations are needed to
set the problem.

Let $\mathcal{M}$ be a compact boundaryless three-dimensional Riemannian
manifold and $m$ the associated volume form. We denote by $\Lambda $ a
singular-hyperbolic attractor of a $C^{2}$ flow $\left( \chi ^{t},\ t\in 
\mathbb{R}\right) $ generated by the vector field $V$ on $\mathcal{M}$ and $%
\mu $ be the physical measure for $\left( \chi ^{t},\ t\in \mathbb{R}\right) 
$ supported on $\Lambda .$ By Theorem B in \cite{APPV} $\mu $ is hyperbolic,
that is for $\mu $-a.e. $x\in \Lambda $ the tangent bundle $T_{x}\mathcal{M}$
splits as three one-dimensional invariant subspaces, namely $E_{x}^{s}\oplus
E_{x}^{V}\oplus E_{x}^{u},$ where $E_{x}^{V}$ is the flow's direction and $%
E_{x}^{u}$ is the direction with positive Lyapunov exponent. This implies
the existence for $\mu $-a.e. $x\in \Lambda $ of the strong-unstable
manifold 
\begin{equation}
W_{x}^{uu}:=\left\{ y\in \mathcal{M}:\lim_{t\rightarrow -\infty }dist\left(
\chi ^{t}\left( y\right) ,\chi ^{t}\left( x\right) \right) =0\right\}
\subset \Lambda \ .
\end{equation}%
By definition of singular-hyperbolic attractor the sub-bundle $E_{\Lambda
}^{s}:=\bigcup\limits_{x\in \Lambda }\left\{ x\right\} \times E_{x}^{s}$ is
uniformly contracting, i.e., for any $t>0$ and any $x\in \Lambda ,\left\Vert
D\chi ^{t}\upharpoonleft _{E_{x}^{x}}\right\Vert <c\kappa ^{t}$ for some $%
c>0,\kappa \in \left( 0,1\right) .$ This implies the existence of the
strong-stable manifold $W_{x}^{ss}$ tangent to $E_{x}^{s}$ at $x.$ Let then,
for any $x\in \Lambda ,W_{x}^{s}:=\bigcup\limits_{t\in \mathbb{R}}\chi
^{t}\left( W_{x}^{ss}\right) $ be the stable manifold of $x$ and, for $\mu $%
-a.e. $x\in \Lambda ,W_{x}^{u}:=\bigcup\limits_{t\in \mathbb{R}}\chi
^{t}\left( W_{x}^{uu}\right) $ be the unstable manifold of $x.$ Given a smooth
cross-section $\Sigma $ of the flow at $x\in \Lambda ,$ let $\omega _{0}$ be
the connected component of $W_{x}^{s}\cap \Sigma $ containing $x.$ Since $%
\omega_{0}$ is a smooth curve, one can take a parametrization $\Psi :\left[
-\epsilon ,\epsilon \right] ^{2}\mapsto \Sigma $ of a compact neighborhood $%
\Sigma _{0}$ of $x$ such that: $\Psi \left( 0,0\right) =x,$ the graph of the
function $\left[ -\epsilon ,\epsilon \right] \ni u\longmapsto \Psi \left(
u,0\right) \in \Sigma $ is a subset of $\omega _{0}$ and the graph of the
function $\left[ -\epsilon ,\epsilon \right] \ni u\longmapsto \Psi \left(
0,u\right) \in \Sigma _{0}$ is a curve $\omega _{1}$ transverse to $\omega
_{0}.$ Thus, given $z\in \Sigma _{0},$ the connected component of $%
W_{z}^{u}\cap \Sigma _{0}$ containing $z,$ which we denote by $\zeta _{z},$
is said to cross $\Sigma _{0}$ if it can be written as the graph of a
function in the coordinate system given by the arc-lengths of $\omega _{0}$
and $\omega _{1}.$ Let $\Pi _{0}:=\left\{ \zeta _{z}:z\in \Sigma
_{0}\right\} .$ Given $\delta >0$ and $\zeta \in \Pi _{0},$ the
two-dimensional surface $\left\{ \chi ^{t}\left( \zeta \right) :t\in \left[
-\delta ,\delta \right] \right\} $ is given by a family of curves each of
which is tangent to $E_{z}^{cu}:=E_{z}^{V}\oplus E_{z}^{u},z\in \zeta ,$
where the map $z\longmapsto E_{z}^{cu}$ is continuous. Then, the collection
of surfaces $\Pi _{\delta }\left( x\right) :=\left\{ \left\{ \chi ^{t}\left(
\zeta \right) :t\in \left[ -\delta ,\delta \right] \right\} :\zeta \in \Pi
_{0}\right\} $ forms a measurable partition of $\bar{\Pi}_{\delta }\left(
x\right) :=\bigcup\limits_{\gamma \in \Pi _{\delta }\left( x\right) }\gamma ,
$ hence there exists a probability kernel $\mathcal{B}\left( \mathcal{M}%
\right) \times \Pi _{\delta }\left( x\right) \ni \left( B,\gamma \right)
\longmapsto \mu \left( B|\gamma \right) \in \left[ 0,1\right] $ such that 
\begin{equation}
\mu \left( B\cap \bar{\Pi}_{\delta }\left( x\right) \right) =\int \mu \left(
B\cap \bar{\Pi}_{\delta }\left( x\right) |\gamma \right) d\bar{\mu}\left(
\gamma \right)   \label{disint}
\end{equation}%
where $\bar{\mu}$ is the measure on $\left( \Pi _{\delta }\left( x\right) ,%
\mathcal{B}\left( \Pi _{\delta }\left( x\right) \right) \right) $ such that $%
\forall A\in \mathcal{B}\left( \Pi _{\delta }\left( x\right) \right) ,\bar{%
\mu}\left( A\right) :=\mu \left( \bigcup\limits_{\gamma \in A}\gamma \right)
.$ By Theorem C in \cite{APPV}, denoting by $m_{\gamma }$ the area form
induced by $m$ on each $\gamma \in \Pi _{\delta }\left( x\right) ,$ it
follows that for $\delta $ sufficiently small, for $\bar{\mu}$-a.e. $\gamma
\in \Pi _{\delta }\left( x\right) ,\mu \left( \cdot |\gamma \right)
<<m_{\gamma }$ and the Radon-Nikodym derivative $\frac{d\mu \left( \cdot
|\gamma \right) }{dm_{\gamma }}$ is bounded from above.

\begin{proposition}
Let $\Lambda $ be a singular-hyperbolic attractor of a $C^{2}$ flow $\left(
\chi ^{t},\ t\in \mathbb{R}\right) $ generated by the vector field $V$ on a
compact boundaryless three-dimensional Riemannian manifold $\mathcal{M}$ and 
$\mu $ be the physical measure for the flow supported on $\Lambda .$ Then,
for any $x\in \Lambda ,$%
\begin{equation}
d_{\mu }^{-}(x)\geq 2\  \label{d-L}
\end{equation}%
and, for all $q\in \mathbb{R},$%
\begin{equation}
D_{q}^{-}(\mu )\geq 2\ .  \label{Dq-L}
\end{equation}
\end{proposition}

\begin{proof}
By (\ref{disint}), for any $x\in \Lambda $ and $\delta >0$ sufficiently
small, since $r\downarrow 0,$%
\begin{eqnarray}
\mu \left( B_{r}^{\left( 3\right) }\left( x\right) \right)  &=&\mu \left(
B_{r}^{\left( 3\right) }\left( x\right) \cap \bar{\Pi}_{\delta }\left(
x\right) \right) =\int \mu \left( B_{r}^{\left( 3\right) }\left( x\right)
\cap \bar{\Pi}_{\delta }\left( x\right) |\gamma \right) d\bar{\mu}\left(
\gamma \right)  \\
&=&\int_{\Pi _{\delta }\left( x\right) }d\bar{\mu}\left( \gamma \right)
m_{\gamma }\left[ \frac{d\mu \left( \cdot |\gamma \right) }{dm_{\gamma }}%
\mathbf{1}_{B_{r}^{\left( 3\right) }\left( x\right) \cap \gamma }\right] \ .
\notag
\end{eqnarray}%
But, since by the definition of $\Pi _{\delta }\left( x\right)
,B_{r}^{\left( 3\right) }\left( x\right) \cap \gamma $ is a compact subset
of the two-dimensional smooth manifold $\gamma ,$ by Theorem C in \cite{APPV}
(see also \cite{AP} section 7.3.10) there exists a constant $c_{1}>0$ such
that $\frac{d\mu \left( \cdot |\gamma \right) }{dm_{\gamma }}<c_{1},$ hence 
\begin{equation}
m_{\gamma }\left[ \frac{d\mu \left( \cdot |\gamma \right) }{dm_{\gamma }}%
\mathbf{1}_{B_{r}^{\left( 3\right) }\left( x\right) \cap \gamma }\right]
\leq c_{1}m_{\gamma }\left( B_{r}^{\left( 3\right) }\left( x\right) \cap
\gamma \right) \leq c_{2}r^{2}
\end{equation}%
for some $c_{2}\geq c_{1},$ which gives (\ref{d-L}). Moreover, $\forall
q\geq 2,$ the same argument leads to 
\begin{equation}
\mu \left[ \mu ^{q-1}\left( B_{r}^{\left( 3\right) }\left( x\right) \right) %
\right] \leq c_{2}^{q-1}r^{2\left( q-1\right) }\ ,
\end{equation}%
which implies (\ref{Dq-L}) for such values of $q.$ Since $D_{q}^{-}$ is a
non-increasing function of $q,$ this result extends to all values of $q\in 
\mathbb{R}.$
\end{proof}

\begin{remark}
We recall that, as put forward in the introduction, $\mu $ is the
push-forward under the diffeomorphism $\Theta $ given in (\ref{Theta}) of the invariant
measure $\mu _{S}:=\mu _{R}\otimes dt\frac{\mathbf{1}_{\left[ 0,\mathfrak{t}%
\right] }}{\mu _{R}\left[ \mathfrak{t}\right] }$ of the suspension flow $S$
built on the Poincar\'{e} map $R$ relative to a cross-section $\Sigma $ with
return time $\mathfrak{t}.$ Then, by Theorem \ref{mainT} we get $d_{\mu
_{R}}^{-}\left( x\right) \geq 1$ and $\forall q\in \mathbb{R}%
,D_{q}^{-}\left( \mu _{R}\right) \geq 1.$
\end{remark}
In particular, for the geometric and the classical Lorenz flows, we get the following result.

\begin{proposition}
$D_{1}(\mu)$ exists and satisfies 
\begin{equation}
D_{1}(\mu)=D_{1}(\mu_R)+1\ .
\end{equation}%
\end{proposition}

\begin{proof}
Since $\mu_{R}$ is exact dimensional \cite{GP}, the result follows by point 2 and 3 of Theorem \ref{mainT}.
\end{proof}

\begin{remark} 
An explicit formula relates $D_{1}(\mu _{R})$ to the entropy and the partial derivatives of $R$ \cite{GP}.
\end{remark}

\section{Appendix}

\subsection*{Proof of Proposition \protect\ref{D_q}}

The proof of Proposition \ref{D_q} relies on the following lemmata.

\begin{lemma}
\label{B-C}Given $r$ sufficiently small, for any $\left( \xi ,t\right) \in {%
\Sigma }_{\mathfrak{t}},$ there exist $s\in \left[ -\frac{\kappa }{2},\frac{%
\kappa }{2}\right] $ and $c>1$ such that 
\begin{equation}
\Theta \left( B_{c^{-1}r}^{\left( d-1\right) }\left( \xi \right) \times
B_{c^{-1}r}^{\left( 1\right) }\left( t+s\right) \right) \subseteq \chi
^{s}\left( B_{r}^{\left( d\right) }\left( \Theta \left( \xi ,t\right)
\right) \right) \subseteq \Theta \left( B_{cr}^{\left( d-1\right) }\left(
\xi \right) \times B_{cr}^{\left( 1\right) }\left( t+s\right) \right)
\subset \mathcal{U}  \label{ballCyl}
\end{equation}%
and 
\begin{equation}
\Theta \left( B_{cr}^{\left( d-1\right) }\left( \xi \right) \times
B_{cr}^{\left( 1\right) }\left( t+s\right) \right) \cap \Sigma =\varnothing
\ .
\end{equation}
\end{lemma}

\begin{proof}
Let $r_{0}\in \left( 0,1\right) $ be sufficiently small such that $\mathcal{U%
}_{r_{0}}\left( \Sigma \right) :=\bigcup\limits_{y\in \Sigma
}B_{r_{0}}^{\left( d\right) }\left( y\right) $ is strictly contained in $%
\mathcal{U}$ and does not contain critical points of the phase velocity
field $V,$ so that $\inf_{y\in \mathcal{U}_{r_{0}}\left( \Sigma \right)
}\left\vert V\left( y\right) \right\vert >0.$ Making use of (\ref{counter}) we set 
\begin{align}
\mathcal{U}_{r_{0}}^{+}\left( \Sigma \right) &:=\left\{ x\in \mathcal{U}_{r_{0}}\left(
\Sigma \right) :\mathfrak{n}\left( \pi _{2}\left( x\right) ,r_{0}+\pi
_{1}\left( x\right) \right) =0\right\} \,,\\
\mathcal{U}_{r_{0}}^{-}\left(\Sigma \right) &:=\left\{ x\in \mathcal{U}_{r_{0}}\left( \Sigma \right) :%
\mathfrak{n}\left( \pi _{2}\left( x\right) ,r_{0}+\pi _{1}\left( x\right)
\right) =1\right\} =\mathcal{U}_{r_{0}}\left( \Sigma \right) \backslash 
\mathcal{U}_{r_{0}}^{+}\left( \Sigma \right) .    
\end{align}
Then, since by construction $\Theta $ is bi-Lipschitz, for all $%
B_{r}^{\left( d\right) }\left( \Theta \left( \xi ,t\right) \right) $ such
that $\Theta \left( \xi ,t\right) \in \mathcal{U}\backslash \mathcal{U}%
_{r_{0}}\left( \Sigma \right) $ and $r<\frac{r_{0}}{2},$ (\ref{ballCyl}) is
clearly satisfied for $s=0.$

Moreover, setting $V_{r_{0}}^{-}:=\max_{i=1,..,d}\inf_{y\in \mathcal{U}%
_{r_{0}}\left( \Sigma \right) }\left\vert V_{i}\left( y\right) \right\vert ,$
for any $y\in \mathcal{U}_{r_{0}}^{+}\left( \Sigma \right) $ and $s>0,$%
\begin{equation}
\left\vert \chi ^{s}\left( y\right) -y\right\vert ^{2}=\sum_{i=1}^{d}\left(
\int_{0}^{s}du\left( \chi _{i}^{u}\left( y\right) \right) \right)
^{2}=\sum_{i=1}^{d}\left( sV_{i}\left( \chi ^{\lambda _{i}s}\left( y\right)
\right) \right) ^{2}
\end{equation}%
for some $\left( \lambda _{1},..,\lambda _{d}\right) \in \left( 0,1\right)
^{d}.$ Therefore, $\left\vert \chi ^{s}\left( y\right) -y\right\vert \geq
sV_{r_{0}}^{-}.$ Hence, denoting by $L:=\sup_{y\in \mathcal{U}}\left\vert
DV\left( y\right) \right\vert $ the Lipschitz constant of $V$ in $\mathcal{U}%
,$ given $x\in \mathcal{U}_{r_{0}}\left( \Sigma \right) ,$ for any $y,z\in
B_{r}^{\left( d\right) }\left( x\right) ,s>0,$ by the Gronwall inequality,%
\begin{equation}
\left\vert \chi ^{s}\left( y\right) -\chi ^{s}\left( z\right) \right\vert
\leq \left\vert x-y\right\vert +\int_{0}^{s}du\left\vert V\left( \chi
^{u}\left( y\right) \right) -V\left( \chi ^{u}\left( z\right) \right)
\right\vert \leq 2re^{sL}\ .
\end{equation}%
Since by definition $\left( 0,1\right) \ni r_{0}\longmapsto V_{r_{0}}^{-}\in %
\left[ V_{1}^{-},V_{0}^{-}\right] ,$ where $V_{0}^{-}:=\max_{i=1,..,d}%
\inf_{y\in \Sigma }\left\vert V_{i}\left( y\right) \right\vert ,$ is a
non-increasing function, $\left( 0,1\right) \ni r_{0}\longmapsto \frac{r_{0}%
}{V_{r_{0}}^{-}}\in \left( 0,\frac{1}{V_{1}^{-}}\right) $ is a
non-decreasing function, therefore we can choose $r_{0}$ sufficiently small
such that $\frac{r_{0}}{V_{r_{0}}^{-}}<\frac{\kappa }{8}$ and for all $s\in
\left( 2\frac{r_{0}}{V_{r_{0}}^{-}},4\frac{r_{0}}{V_{r_{0}}^{-}}\right)
,dist\left( \chi ^{s}\left( B_{r}^{\left( d\right) }\left( x\right) \right)
,\Sigma \right) >2r_{0}$ and $diam\left( \chi ^{s}\left( B_{r}^{\left(
d\right) }\left( x\right) \right) \right) <2re^{4r_{0}\frac{L}{V_{r_{0}}^{-}}%
}.$ Hence, given $\left( \xi ,t\right) \in \Sigma _{\mathfrak{t}}$ such that 
$\Theta \left( \xi ,t\right) \in \mathcal{U}_{r_{0}}^{+}\left( \Sigma
\right) ,$ since by construction $s<\frac{\kappa }{2},\mathfrak{n}\left( \xi
,s+t\right) =0$ and for $r<\frac{r_{0}}{2}e^{-4r_{0}\frac{L}{V_{r_{0}}^{-}}%
}, $ setting $C:=e^{4r_{0}\frac{L}{V_{r_{0}}^{-}}},$%
\begin{gather}
\chi ^{s}\left( B_{r}^{\left( d\right) }\left( \Theta \left( \xi ,t\right)
\right) \right) \subseteq B_{rC}^{\left( d\right) }\left( \left( \chi
^{s}\circ \Theta \right) \left( \xi ,t\right) \right) =B_{rC}^{\left(
d\right) }\left( \left( \Theta \circ S^{s}\right) \left( \xi ,t\right)
\right)  \label{UBballcyl} \\
=B_{rC}^{\left( d\right) }\left( \Theta \left( R^{\mathfrak{n}\left( \xi
,s+t\right) }\left( \xi \right) ,\left( s+t-\mathfrak{s}_{\mathfrak{n}\left(
\xi ,s+t\right) }\right) \right) \right) =B_{rC}^{\left( d\right) }\left(
\Theta \left( \xi ,t+s\right) \right) \subseteq  \notag \\
\subseteq \Theta \left( B_{rcC}^{\left( d-1\right) }\left( \xi \right)
\times B_{rcC}^{\left( 1\right) }\left( t+s\right) \right) \ .  \notag
\end{gather}%
On the other hand, since $\forall y,z\in B_{r}^{\left( d\right) }\left(
x\right) ,$ and $t>s\geq 0,$%
\begin{gather}
\left\vert \chi ^{t}\left( y\right) -\chi ^{t}\left( z\right) \right\vert
-\left\vert \chi ^{s}\left( y\right) -\chi ^{s}\left( z\right) \right\vert
=\left\vert \left( y-z\right) +\int_{0}^{s}du\left( V_{i}\left( \chi
^{u}\left( y\right) \right) -V_{i}\left( \chi ^{u}\left( z\right) \right)
\right) +\right. \\
\left. +\int_{s}^{t}du\left( V_{i}\left( \chi ^{u}\left( y\right) \right)
-V_{i}\left( \chi ^{u}\left( z\right) \right) \right) \right\vert
-\left\vert \left( y-z\right) +\int_{0}^{s}du\left( V_{i}\left( \chi
^{u}\left( y\right) \right) -V_{i}\left( \chi ^{u}\left( z\right) \right)
\right) \right\vert \geq  \notag \\
\geq -L\int_{s}^{t}du\left\vert \chi _{i}^{u}\left( y\right) -\chi
_{i}^{u}\left( z\right) \right\vert \ ,  \notag
\end{gather}%
we get 
\begin{equation}
\left\{ 
\begin{array}{l}
\frac{d}{ds}\left\vert \chi ^{s}\left( y\right) -\chi ^{s}\left( z\right)
\right\vert \geq -L\left\vert \chi ^{s}\left( y\right) -\chi ^{s}\left(
z\right) \right\vert \\ 
\left\vert \chi ^{s}\left( y\right) -\chi ^{s}\left( z\right) \right\vert
\upharpoonleft _{s=0}=\left\vert y-z\right\vert%
\end{array}%
\right. \ ,
\end{equation}%
which implies $\left\vert \chi ^{s}\left( y\right) -\chi ^{s}\left( z\right)
\right\vert \geq \left\vert y-z\right\vert e^{-Ls}.$ Hence, if $\left( \xi
,t\right) \in \Sigma _{\mathfrak{t}}$ such that $\Theta \left( \xi ,t\right)
\in \mathcal{U}_{r_{0}}^{+}\left( \Sigma \right) ,$ there exists $c^{\prime
}\in (0,1]$ such that, for all $s\in \left( 2\frac{r_{0}}{V_{r_{0}}^{-}},4%
\frac{r_{0}}{V_{r_{0}}^{-}}\right) ,$%
\begin{gather}
\chi ^{s}\left( B_{r}^{\left( d\right) }\left( \Theta \left( \xi ,t\right)
\right) \right) \supseteq B_{rc^{\prime }C^{-1}}^{\left( d\right) }\left(
\left( \chi ^{s}\circ \Theta \right) \left( \xi ,t\right) \right)
=B_{rc^{\prime }C^{-1}}^{\left( d\right) }\left( \left( \Theta \circ
S^{s}\right) \left( \xi ,t\right) \right)  \label{LBballcyl} \\
=B_{rc^{\prime }C^{-1}}^{\left( d\right) }\left( \Theta \left( R^{\mathfrak{n%
}\left( \xi ,s+t\right) }\left( \xi \right) ,\left( s+t-\mathfrak{s}_{%
\mathfrak{n}\left( \xi ,s+t\right) }\right) \right) \right) =B_{rc^{\prime
}C^{-1}}^{\left( d\right) }\left( \Theta \left( \xi ,t+s\right) \right)
\supseteq  \notag \\
\supseteq \Theta \left( B_{rc^{.1}c^{\prime }C^{-1}}^{\left( d-1\right)
}\left( \xi \right) \times B_{rc^{.1}c^{\prime }C^{-1}}^{\left( 1\right)
}\left( t+s\right) \right) \ .  \notag
\end{gather}%
Thus, setting $K:=cC\left( c^{\prime }\right) ^{-1},$%
\begin{gather}
\Theta \left( B_{rK^{-1}}^{\left( d-1\right) }\left( \xi \right) \times
B_{rK^{-1}}^{\left( 1\right) }\left( t+s\right) \right) \subseteq \chi
^{s}\left( B_{r}^{\left( d\right) }\left( \Theta \left( \xi ,t\right)
\right) \right) \subseteq \\
\subseteq \Theta \left( B_{rK}^{\left( d-1\right) }\left( \xi \right) \times
B_{rK}^{\left( 1\right) }\left( t+s\right) \right)  \notag
\end{gather}%
provided $r$ is chosen sufficiently small.

On the other hand, if $\left( \xi ,t\right) \in \Sigma _{\mathfrak{t}}$ such
that $\Theta \left( \xi ,t\right) \in \mathcal{U}_{r_{0}}^{-}\left( \Sigma
\right) ,$ from (\ref{UBballcyl}) and (\ref{LBballcyl}) we get 
\begin{equation}
\chi ^{s}\left( B_{rK^{-1}}^{\left( d\right) }\left( \Theta \left( \xi
,t-s\right) \right) \right) \subseteq B_{r}^{\left( d\right) }\left( \Theta
\left( \xi ,t\right) \right) \subseteq \chi ^{s}\left( B_{rK}^{\left(
d\right) }\left( \Theta \left( \xi ,t-s\right) \right) \right) \ ,
\end{equation}%
which implies 
\begin{equation}
\Theta \left( B_{rK^{-1}}^{\left( d-1\right) }\left( \xi \right) \times
B_{rK^{-1}}^{\left( 1\right) }\left( t-s\right) \right) \subseteq \chi
^{-s}\left( B_{r}^{\left( d\right) }\left( \Theta \left( \xi ,t\right)
\right) \right) \subseteq \Theta \left( B_{rK}^{\left( d-1\right) }\left(
\xi \right) \times B_{rK}^{\left( 1\right) }\left( t-s\right) \right) \ .
\end{equation}
\end{proof}

\begin{lemma}
\label{Ld4}Let 
\begin{equation}
\nu _{R}\left( d\xi \right) :=\mu _{R}\left( d\xi \right) \frac{\mathfrak{t}%
\left( \xi \right) }{\mu _{R}\left[ \mathfrak{t}\right] }\ .
\end{equation}%
Then, for $q\neq 1$ 
\begin{equation}
D_{q}\left( \nu _{R}\right) =D_{q}\left( \mu _{R}\right) \ .
\end{equation}
\end{lemma}

\begin{proof}
Let $q>1.$ Since there exists $\kappa >0$ such that, $\kappa \leq \mathfrak{%
t\leq }\frac{1}{\kappa },$ we have 
\begin{align}
\int_{\Sigma }\nu _{R}\left( d\xi \right) \nu _{R}^{q-1}\left( B_{r}^{\left(
d-1\right) }\left( \xi \right) \right) & \leq \frac{1}{\left\{ \kappa \mu
_{R}\left[ \mathfrak{t}\right] \right\} ^{q}}\int_{\Sigma }\mu _{R}\left(
d\xi \right) \mu _{R}^{q-1}\left( B_{r}^{\left( d-1\right) }\left( \xi
\right) \right) \ , \\
\int_{\Sigma }\nu _{R}\left( d\xi \right) \nu _{R}^{q-1}\left( B_{r}^{\left(
d-1\right) }\left( \xi \right) \right) & \geq \left\{ \frac{\kappa }{\mu _{R}%
\left[ \mathfrak{t}\right] }\right\} ^{q}\int_{\Sigma }\mu _{R}\left( d\xi
\right) \mu _{R}^{q-1}\left( B_{r}^{\left( d-1\right) }\left( \xi \right)
\right) \ .
\end{align}

Similarly, if $q<1,$%
\begin{gather}
\int_{\Sigma }\nu _{R}\left( d\xi \right) \nu _{R}^{q-1}\left( B_{r}^{\left(
d-1\right) }\left( \xi \right) \right) =\int_{\Sigma }\nu _{R}\left( d\xi
\right) \frac{1}{\nu _{R}^{1-q}\left( B_{r}^{\left( d-1\right) }\left( \xi
\right) \right) } \\
\leq \frac{1}{\kappa \mu _{R}\left[ \mathfrak{t}\right] }\frac{\mu _{R}^{1-q}%
\left[ \mathfrak{t}\right] }{\kappa ^{1-q}}\int_{\Sigma }\mu _{R}\left( d\xi
\right) \frac{1}{\mu _{R}^{1-q}\left( B_{r}^{\left( d-1\right) }\left( \xi
\right) \right) }  \notag \\
=\left\{ \frac{\kappa }{\mu _{R}\left[ \mathfrak{t}\right] }\right\}
^{q}\int_{\Sigma }\mu _{R}\left( d\xi \right) \mu _{R}^{q-1}\left(
B_{r}^{\left( d-1\right) }\left( \xi \right) \right) \ ,  \notag \\
\int_{\Sigma }\nu _{R}\left( d\xi \right) \frac{1}{\nu _{R}^{1-q}\left(
B_{r}^{\left( d-1\right) }\left( \xi \right) \right) }\geq \frac{1}{\left\{
\kappa \mu _{R}\left[ \mathfrak{t}\right] \right\} ^{q}}\int_{\Sigma }\mu
_{R}\left( d\xi \right) \mu _{R}^{q-1}\left( B_{r}^{\left( d-1\right)
}\left( \xi \right) \right) \ .
\end{gather}
\end{proof}

\begin{lemma}
\label{Ld4bis}For $q>1,$%
\begin{eqnarray}
\mu \left[ \mu ^{q-1}\left( B_{r}^{\left( d\right) }\left( \cdot \right)
\right) \right] &\leq &\left( 2cr\right) ^{q-1}\int_{\Sigma }\mu _{R}\left(
d\xi \right) \frac{\mathfrak{t}\left( \xi \right) }{\mu _{R}^{q}\left[ 
\mathfrak{t}\right] }\mu _{R}^{q-1}\left( B_{cr}^{\left( d-1\right) }\left(
\xi \right) \right) \ ,  \label{q>1ub} \\
\mu \left[ \mu ^{q-1}\left( B_{r}^{\left( d\right) }\left( \cdot \right)
\right) \right] &\geq &\left( \frac{2}{c}r\right) ^{q-1}\int_{\Sigma }\mu
_{R}\left( d\xi \right) \frac{\mathfrak{t}\left( \xi \right) }{\mu _{R}^{q}%
\left[ \mathfrak{t}\right] }\mu _{R}^{q-1}\left( B_{c^{-1}r}^{\left(
d-1\right) }\left( \xi \right) \right) \ ,  \label{q>1lb}
\end{eqnarray}%
while, for $q<1,$%
\begin{eqnarray}
\mu \left[ \mu ^{q-1}\left( B_{r}^{\left( d\right) }\left( \cdot \right)
\right) \right] &\leq &\left( \frac{2}{c}r\right) ^{q-1}\int_{\Sigma }\mu
_{R}\left( d\xi \right) \frac{\mathfrak{t}\left( \xi \right) }{\mu _{R}^{q}%
\left[ \mathfrak{t}\right] }\mu _{R}^{q-1}\left( B_{c^{-1}r}^{\left(
d-1\right) }\left( \xi \right) \right) \ ,  \label{q<1ub} \\
\mu \left[ \mu ^{q-1}\left( B_{r}^{\left( d\right) }\left( \cdot \right)
\right) \right] &\geq &\left( 2cr\right) ^{q-1}\int_{\Sigma }\mu _{R}\left(
d\xi \right) \frac{\mathfrak{t}\left( \xi \right) }{\mu _{R}^{q}\left[ 
\mathfrak{t}\right] }\mu _{R}^{q-1}\left( B_{cr}^{\left( d-1\right) }\left(
\xi \right) \right) \ .  \label{q<1lb}
\end{eqnarray}
\end{lemma}

\begin{proof}
Let $r_{0}\in \left( 0,1\right) $ be such that Lemma \ref{B-C} is in force.
Since $\mu $ is invariant for $\left( \chi ^{t},t\in \mathbb{R}\right) ,$
for any $q\neq 1$ and any $x\in \mathcal{U},$%
\begin{equation}
\mu ^{q-1}\left( B_{r}^{\left( d\right) }\left( x\right) \right) =\mu
^{q-1}\left( \chi ^{u}\left( B_{r}^{\left( d\right) }\left( x\right) \right)
\right) \;,\;u\in \mathbb{R}\ .
\end{equation}%
Moreover, 
\begin{align}
\mu \left[ \mu ^{q-1}\left( B_{r}^{\left( d\right) }\left( \cdot \right)
\right) \right] & = \mu \left[ \mu ^{q-1}\left( B_{r}^{\left( d\right) }\left(
\cdot \right) \right) \mathbf{1}_{\mathcal{U}\backslash \mathcal{U}_{0}}%
\right] +\mu \left[ \mu ^{q-1}\left( B_{r}^{\left( d\right) }\left( \cdot
\right) \right) \mathbf{1}_{\mathcal{U}_{0}^{+}}\right] + \notag \\
&+\mu \left[ \mu
^{q-1}\left( B_{r}^{\left( d\right) }\left( \cdot \right) \right) \mathbf{1}%
_{\mathcal{U}_{0}^{-}}\right] \ .
\end{align}%
Hence, adjusting the definition of $\mathcal{U}_{0}^{+}$ in such a way that $%
\mathcal{V}\ni \left( \xi ,t\right) \longmapsto \mathbf{1}_{\mathcal{U}%
_{0}^{+}}\circ \Theta \left( \xi ,t\right) =\mathbf{1}_{\Sigma }\left( \xi
\right) \times \mathbf{1}_{\left[ 0,r_{0}\right] }\left( t\right) $ and $%
\mathcal{V}\ni \left( \xi ,t\right) \longmapsto \mathbf{1}_{\mathcal{U}%
_{0}^{-}}\circ \Theta \left( \xi ,t\right) =\mathbf{1}_{\Sigma }\left( \xi
\right) \times \mathbf{1}_{\left[ \mathfrak{t}\left( \xi \right) -r_{0},%
\mathfrak{t}\left( \xi \right) \right] }\left( t\right) ,$ for $q>1,$ by
Lemma \ref{B-C}, we can choose $u\in \left[ 0,\frac{\kappa }{2}\right] $
such that 
\begin{align}
\mu \left[ \mu ^{q-1}\left( B_{r}^{\left( d\right) }\left( \cdot \right)
\right) \mathbf{1}_{\mathcal{U}_{0}^{+}}\right] & =\mu \left[ \mu
^{q-1}\left( \chi ^{u}\left( B_{r}^{\left( d\right) }\left( \cdot \right)
\right) \right) \mathbf{1}_{\mathcal{U}_{0}^{+}}\right] \\
& \leq \mu \left[ \mu ^{q-1}\Theta \left( B_{cr}^{\left( d-1\right) }\left(
\pi _{2}\circ \Theta ^{-1}\left( \cdot \right) \right) \times B_{cr}^{\left(
1\right) }\left( \pi _{1}\circ \Theta ^{-1}\left( \cdot \right) +u\right)
\right) \mathbf{1}_{\mathcal{U}_{0}^{+}}\right]  \notag \\
& =\Theta _{\ast }\mu _{S}\left[ \left( \Theta _{\ast }\mu _{S}\right)
^{q-1}\Theta \left( B_{cr}^{\left( d-1\right) }\left( \pi _{2}\circ \Theta
^{-1}\left( \cdot \right) \right) \times B_{cr}^{\left( 1\right) }\left( \pi
_{1}\circ \Theta ^{-1}\left( \cdot \right) +u\right) \right) \mathbf{1}_{%
\mathcal{U}_{0}^{+}}\right]  \notag \\
& =\mu _{S}\left[ \left( \Theta _{\ast }\mu _{S}\right) ^{q-1}\left( \Theta
\left( B_{cr}^{\left( d-1\right) }\left( \pi _{2}\left( \cdot \right)
\right) \times B_{cr}^{\left( 1\right) }\left( \pi _{1}\left( \cdot \right)
+u\right) \right) \right) \mathbf{1}_{\mathcal{U}_{0}^{+}}\circ \Theta %
\right]  \notag \\
& =\mu _{S}\left[ \mu _{S}^{q-1}\left( B_{cr}^{\left( d-1\right) }\left( \pi
_{2}\left( \cdot \right) \right) \times B_{cr}^{\left( 1\right) }\left( \pi
_{1}\left( \cdot \right) +u\right) \right) \mathbf{1}_{\mathcal{U}%
_{0}^{+}}\circ \Theta \right]  \notag \\
& =\int_{\Sigma }\mu _{R}\left( d\xi \right) \frac{\int_{0}^{\mathfrak{t}%
\left( \xi \right) }dt}{\mu _{R}\left[ \mathfrak{t}\right] }\left\{
\int_{\Sigma }\mu _{R}\left( d\eta \right) \frac{\int_{0}^{\mathfrak{t}%
\left( \eta \right) }ds}{\mu _{R}\left[ \mathfrak{t}\right] }\mathbf{1}%
_{B_{cr}^{\left( 1\right) }\left( t+u\right) }\left( s\right) \mathbf{1}%
_{B_{cr}^{\left( d-1\right) }\left( \xi \right) }\left( \eta \right)
\right\} ^{q-1}\mathbf{1}_{\left[ 0,r_{0}\right] }\left( t\right)  \notag
\end{align}%
for some $c>1$ and $r$ sufficiently small. Since by construction $%
B_{cr}^{\left( 1\right) }\left( t+u\right) \subset \left[ 0,\kappa \right] ,$%
\begin{equation}
\mu \left[ \mu ^{q-1}\left( B_{r}^{\left( d\right) }\left( \cdot \right)
\right) \mathbf{1}_{\mathcal{U}_{0}^{+}}\right] \leq \left( 2cr\right)
^{q-1}\int_{\Sigma }\mu _{R}\left( d\xi \right) \frac{\int_{0}^{\mathfrak{t}%
\left( \xi \right) }dt}{\mu _{R}^{q}\left[ \mathfrak{t}\right] }\mu
_{R}^{q-1}\left( B_{cr}^{\left( d-1\right) }\left( \xi \right) \right) 
\mathbf{1}_{\left[ 0,r_{0}\right] }\left( t\right) \ .  \label{q>1ub+}
\end{equation}%
On the other hand, 
\begin{gather}
\mu \left[ \mu ^{q-1}\left( B_{r}^{\left( d\right) }\left( \cdot \right)
\right) \mathbf{1}_{\mathcal{U}_{0}^{-}}\right] =\mu \left[ \mu
^{q-1}\left( \chi ^{-u}\left( B_{r}^{\left( d\right) }\left( \cdot \right)
\right) \right) \mathbf{1}_{\mathcal{U}_{0}^{-}}\right] \\
\leq \int_{\Sigma }\mu _{R}\left( d\xi \right) \frac{\int_{0}^{\mathfrak{t}%
\left( \xi \right) }dt}{\mu _{R}\left[ \mathfrak{t}\right] }\left\{
\int_{\Sigma }\mu _{R}\left( d\eta \right) \frac{\int_{0}^{\mathfrak{t}%
\left( \eta \right) }ds}{\mu _{R}\left[ \mathfrak{t}\right] }\mathbf{1}%
_{B_{cr}^{\left( 1\right) }\left( t-u\right) }\left( s\right) \mathbf{1}%
_{B_{cr}^{\left( d-1\right) }\left( \xi \right) }\left( \eta \right)
\right\} ^{q-1}\mathbf{1}_{\left[ \mathfrak{t}\left( \xi \right) -r_{0},%
\mathfrak{t}\left( \xi \right) \right] }\left( t\right) \ ,  \notag
\end{gather}%
so that 
\begin{equation}
\mu \left[ \mu ^{q-1}\left( B_{r}^{\left( d\right) }\left( \cdot \right)
\right) \mathbf{1}_{\mathcal{U}_{0}^{+}}\right] \leq \left( 2cr\right)
^{q-1}\int_{\Sigma }\mu _{R}\left( d\xi \right) \frac{\int_{0}^{\mathfrak{t}%
\left( \xi \right) }dt}{\mu _{R}^{q}\left[ \mathfrak{t}\right] }\mu
_{R}^{q-1}\left( B_{cr}^{\left( d-1\right) }\left( \xi \right) \right) 
\mathbf{1}_{\left[ \mathfrak{t}\left( \xi \right) -r_{0},\mathfrak{t}\left(
\xi \right) \right] }\left( t\right) \ .  \label{q>1ub-}
\end{equation}%
Furthermore, 
\begin{eqnarray}
\mu \left[ \mu ^{q-1}\left( B_{r}^{\left( d\right) }\left( \cdot \right)
\right) \mathbf{1}_{\mathcal{U}\backslash \mathcal{U}_{0}}\right] 
&\leq & \int_{\Sigma }\mu _{R}\left( d\xi \right) \frac{\int_{0}^{\mathfrak{t}%
\left( \xi \right) }dt}{\mu _{R}\left[ \mathfrak{t}\right] }\left\{
\int_{\Sigma }\mu _{R}\left( d\eta \right) \frac{\int_{0}^{\mathfrak{t}%
\left( \eta \right) }ds}{\mu _{R}\left[ \mathfrak{t}\right] }\mathbf{1}%
_{B_{cr}^{\left( 1\right) }\left( t\right) }\left( s\right) 
\mathbf{1}_{B_{cr}^{\left( d-1\right) }\left( \xi \right) }\left( \eta \right)\right\}^{q-1} \times  \label{q>1ub_res} \\
&\times & \mathbf{1}_{\left[ r_{0},\mathfrak{t}\left( \xi \right) -r_{0}
\right] }\left( t\right)  \notag \\
&\leq & \left( 2cr\right) ^{q-1}\int_{\Sigma }\mu _{R}\left( d\xi \right) 
\frac{\int_{0}^{\mathfrak{t}\left( \xi \right) }dt}{\mu _{R}^{q}\left[ 
\mathfrak{t}\right] }\mu _{R}^{q-1}\left( B_{cr}^{\left( d-1\right) }\left(
\xi \right) \right) \mathbf{1}_{\left[ r_{0},\mathfrak{t}\left( \xi \right)
-r_{0}\right] }\left( t\right) \ .  \notag
\end{eqnarray}%
Thus, summing (\ref{q>1ub+}), (\ref{q>1ub-}) and (\ref{q>1ub_res}), we get (%
\ref{q>1ub}). (\ref{q>1lb}) follows in a similar way, as well as (\ref{q<1ub}%
) and (\ref{q<1lb}) just noting that, for $q<1,$%
\begin{equation}
\mu \left[ \mu ^{q-1}\left( B_{r}^{\left( d\right) }\left( \cdot \right)
\right) \right] =\mu \left[ \frac{1}{\mu ^{1-q}\left( B_{r}^{\left( d\right)
}\left( \cdot \right) \right) }\right] \ .
\end{equation}
\end{proof}

\begin{lemma}
\label{Ld3}For $q\neq 1,$%
\begin{align}
& \lim_{r\downarrow 0}\frac{\log \left( 2cr\right) ^{q-1}\int_{\Sigma }\mu
_{R}\left( d\xi \right) \frac{\mathfrak{t}\left( \xi \right) }{\mu _{R}^{q}%
\left[ \mathfrak{t}\right] }\mu _{R}^{q-1}\left( B_{cr}^{\left( d-1\right)
}\left( \xi \right) \right) }{\log r} \\
& =\lim_{r\downarrow 0}\frac{\log \left( \frac{2}{c}r\right)
^{q-1}\int_{\Sigma }\mu _{R}\left( d\xi \right) \frac{\mathfrak{t}\left( \xi
\right) }{\mu _{R}^{q}\left[ \mathfrak{t}\right] }\mu _{R}^{q-1}\left(
B_{c^{-1}r}^{\left( d-1\right) }\left( \xi \right) \right) }{\log r}  \notag
\\
& =\left( q-1\right) D_{q}\left( \nu _{R}\right) +q-1\ .  \notag
\end{align}
\end{lemma}

\begin{proof}
For $q>1$%
\begin{align}
& \left( 2cr\right) ^{q-1}\int_{\Sigma }\mu _{R}\left( d\xi \right) \frac{%
\mathfrak{t}\left( \xi \right) }{\mu _{R}^{q}\left[ \mathfrak{t}\right] }\mu
_{R}^{q-1}\left( B_{cr}^{\left( d-1\right) }\left( \xi \right) \right) \\
& \leq \left( \frac{2cr}{\kappa }\right) ^{q-1}\int_{\Sigma }\nu _{R}\left(
d\xi \right) \nu _{R}^{q-1}\left( B_{cr}^{\left( d-1\right) }\left( \xi
\right) \right)  \notag
\end{align}%
and 
\begin{align}
& \left( \frac{2}{c}r\right) ^{q-1}\int_{\Sigma }\mu _{R}\left( d\xi \right) 
\frac{\mathfrak{t}\left( \xi \right) }{\mu _{R}^{q}\left[ \mathfrak{t}\right]
}\mu _{R}^{q-1}\left( B_{c^{-1}r}^{\left( d-1\right) }\left( \xi \right)
\right) \\
& \geq \left( 2cr\kappa \right) ^{q-1}\int_{\Sigma }\nu _{R}\left( d\xi
\right) \nu _{R}^{q-1}\left( B_{cr}^{\left( d-1\right) }\left( \xi \right)
\right) \ ,  \notag
\end{align}%
as well as, for $q<1,$%
\begin{align}
& \left( \frac{2}{c}r\right) ^{q-1}\int_{\Sigma }\mu _{R}\left( d\xi \right) 
\frac{\mathfrak{t}\left( \xi \right) }{\mu _{R}^{q}\left[ \mathfrak{t}\right]
}\mu _{R}^{q-1}\left( B_{c^{-1}r}^{\left( d-1\right) }\left( \xi \right)
\right) \\
& \leq \left( \frac{2c^{-1}r}{\kappa }\right) ^{q-1}\int_{\Sigma }\nu
_{R}\left( d\xi \right) \nu _{R}^{q-1}\left( B_{c^{-1}r}^{\left( d-1\right)
}\left( \xi \right) \right) \ ,  \notag \\
& \left( 2cr\right) ^{q-1}\int_{\Sigma }\mu _{R}\left( d\xi \right) \frac{%
\mathfrak{t}\left( \xi \right) }{\mu _{R}^{q}\left[ \mathfrak{t}\right] }\mu
_{R}^{q-1}\left( B_{cr}^{\left( d-1\right) }\left( \xi \right) \right) \\
& \geq \left( 2c^{-1}r\kappa \right) ^{q-1}\int_{\Sigma }\nu _{R}\left( d\xi
\right) \nu _{R}^{q-1}\left( B_{c^{-1}r}^{\left( d-1\right) }\left( \xi
\right) \right) \ .  \notag
\end{align}\end{proof}

In view of Lemmata \ref{Ld4} and \ref{Ld4bis}, for $q\neq 1,$ Proposition \ref{D_q} follows
from Lemma \ref{Ld3}.

\subsection*{Proof of Lemma \protect\ref{lemm}}

Let $g$ be a positive symmetric mollifier, i.e. a smooth function such that:

\begin{itemize}
\item it is supported on $J;$

\item if $x_{J}$ is the middle point of $J,\int dxg\left( x-x_{J}\right) =1;$

\item $\forall f\in BL\left( J\right) ,$ setting $$J\ni x\longmapsto
g_{\sigma}\left( y\right) :=\sigma^{-1}g_{\sigma}\left( \frac{y-y_{J}}{\sigma%
}\right) \in\mathbb{R}_{+},$$ we have $$\lim_{\sigma\downarrow0}\int dx\frac{g\left( 
\frac{x-x_{J}}{\sigma}\right) }{\sigma}f\left( y\right) =f\left(
x_{J}\right).$$
\end{itemize}
For example we may choose 
\begin{equation}
J\ni y\longmapsto g\left( x\right) :=\frac{Z_{J}}{\sqrt{2\pi }}e^{-\frac{%
\left( x-x_{J}\right) ^{2}}{2}}\mathbf{1}_{J}\left( \left\vert
x-x_{J}\right\vert \right) \in \mathbb{R}_{+}\,
\end{equation}%
with $Z_{J}$ such that $\int_{J}dxg\left( x\right) =1.$ Then, $\forall f\in
BL\left( I\times J\right) $ and $\left( x,y\right) \in I\times J,$ by the
Schwartz inequality, we get 
\begin{align}
\left\vert \int_{J}dzg_{\sigma }\left( z-y\right) f\left( x,z\right)
-f\left( x,y\right) \right\vert & =\int_{J}dzg_{\sigma }\left( z-y\right)
\left\vert f\left( x,z\right) -f\left( x,y\right) \right\vert  \\
& \leq \int_{J}dzg_{\sigma }\left( z-y\right) \left\vert z-y\right\vert  
\notag \\
& \leq \sqrt{\int_{J}dzg_{\sigma }\left( z-y\right) \left( z-y\right) ^{2}}%
\leq Z_{J}\sigma \ ,  \notag
\end{align}%
where we have used that $\forall \left( x,y\right) \in I\times J,$%
\begin{equation}
\left\vert f\left( x,z\right) -f\left( x,y\right) \right\vert \leq
\left\Vert f\right\Vert _{BL\left( I\times J\right) }\left\vert
z-y\right\vert =\left\vert z-y\right\vert \ .
\end{equation}%
Hence, setting $\int_{J}dzg_{\sigma }\left( z-y\right) f\left( x,z\right)
=:\left( g_{\sigma }\ast f\right) \left( x,y\right) $ and considering the
bounded linear operators on the Banach space $BM\left( I\times J\right) $ of
bounded $\mathcal{B}\left( I\times J\right) $-measurable functions endowed
with the sup-norm 
\begin{align}
BM\left( I\times J\right) & \ni f\longmapsto T_{\varphi }f\left( \cdot
,\cdot \right) :=f\left( \cdot ,\varphi \left( \cdot \right) \right) \in
BM\left( I\times J\right) \ , \\
BM\left( I\times J\right) & \ni f\longmapsto G_{\sigma }f:=\left( g_{\sigma
}\ast f\right) \in BM\left( I\times J\right) \ ,
\end{align}%
for any $f\in BM\left( I\times J\right) $ we have 
\begin{equation}
T_{\varphi }\left( G_{\sigma }f\right) =\left( T_{\varphi }G_{\sigma
}\right) f=\left( T_{\varphi }g_{\sigma }\right) \ast f=g_{\sigma }\ast
\left( T_{\varphi }f\right) =G_{\sigma }\left( T_{\varphi }f\right) \ .
\end{equation}

Since $\forall f\in BL\left( I\times J\right) $ 
\begin{align}
& \left\vert \frac{1}{n}\sum_{k=1}^{n}f\circ R^{k}-\frac{1}{n}%
\sum_{k=1}^{n}\left( g_{\sigma}\ast f\right) \left( T^{k}\left( x\right)
,\varphi\circ T^{k-1}\left( x\right) \right) \right\vert \\
& \leq\frac{1}{n}\sum_{k=1}^{n}\left\vert f\left( T^{k}\left( x\right)
,\varphi\circ T^{k-1}\left( x\right) \right) -\left( g_{\sigma}\ast f\right)
\left( T^{k}\left( x\right) ,\varphi\circ T^{k-1}\left( x\right) \right)
\right\vert  \notag \\
& \leq\frac{1}{n}\sum_{k=1}^{n}\int_{J}dzg_{\sigma}\left( z-\varphi\circ
T^{k-1}\left( y\right) \right) \left\vert f\left( T^{k}\left( x\right)
,z\right) -f\left( T^{k}\left( x\right) ,\varphi\circ T^{k-1}\left( x\right)
\right) \right\vert  \notag \\
& \leq Z_{J}\sigma\ ,  \notag
\end{align}
we get 
\begin{align}
\frac{1}{n}\sum_{k=1}^{n}f\circ R^{k} & =\frac{1}{n}\sum_{k=1}^{n}f\left(
T^{k}\left( x\right) ,\varphi\circ T^{k-1}\left( x\right) \right) -\left(
g_{\sigma}\ast f\right) \left( T^{k}\left( x\right) ,\varphi\circ
T^{k-1}\left( x\right) \right) \\
& +\frac{1}{n}\sum_{k=1}^{n}\left( g_{\sigma}\ast f\right) \left(
T^{k}\left( x\right) ,\varphi\circ T^{k-1}\left( x\right) \right)  \notag \\
& \leq Z_{J}\sigma+\frac{1}{n}\sum_{k=1}^{n}\left( g_{\sigma}\ast f\right)
\left( T^{k}\left( x\right) ,\varphi\circ T^{k-1}\left( x\right) \right) 
\notag
\end{align}
and 
\begin{equation}
\frac{1}{n}\sum_{k=1}^{n}f\circ R^{k}\geq-Z_{J}\sigma+\frac{1}{n}\sum
_{k=1}^{n}\left( g_{\sigma}\ast f\right) \left( T^{k}\left( x\right)
,\varphi\circ T^{k-1}\left( x\right) \right) \ .
\end{equation}
But, 
\begin{align}
& \frac{1}{n}\sum_{k=1}^{n}\left( g_{\sigma}\ast f\right) \left( T^{k}\left(
x\right) ,\varphi\circ T^{k-1}\left( x\right) \right) = \\
& =\int_{J}dz\frac{1}{n}\sum_{k=1}^{n}g_{\sigma}\left( z-T^{k-1}\left(
x\right) \right) f\left( T^{k}\left( x\right) ,\varphi\left( z\right)
\right) \underset{n\uparrow\infty}{\longrightarrow}\int_{J}dz\int_{I}\mu
_{T}\left( dy\right) g_{\sigma}\left( z-y\right) f\left( y,\varphi\left(
z\right) \right) \ .  \notag
\end{align}
Letting $\sigma\downarrow0$ we obtain $\lim_{n\rightarrow\infty}\frac{1}{n}%
\sum_{k=1}^{n}f\circ R^{k}=\mu_{T}\left[ f\left( \cdot,\varphi\left(
\cdot\right) \right) \right] \ dxdy-a.s.,$ that is $R$ admits an invariant
measure $\mu_{R}$ whose support is the graph of $\varphi;$ in other words $%
spt\mu_{R}=\left\{ \left( x,y\right) \in I\times J:y=\varphi\left( x\right)
\right\} ,$ and disintegrates as $\mu_{T}\left( dx\right) \delta_{\left\{
\varphi\left( x\right) \right\} }\left( dy\right) $ where $\delta_{\left\{
z\right\} }\left( dy\right) $ stands for the Dirac mass at $\left\{
z\right\} .$

We remark that, as it clearly appears from the proof, this result is
independent of the choice of the mollifier $g.$

\subsection*{Proof of Theorem \protect\ref{propp}}

We will first prove the following result, which deals with the case of a
finite number of singularities, and then move to the infinite case.

\begin{proposition}
\label{propp0} Let $I\subset \mathbb{R}$ an interval, $n\geq 1$ and $\nu $ a
probability measure on $\left( I,\mathcal{B}\left( I\right) \right) $ that
is absolutely continuous with respect to Lebesgue with a density $\rho
_{n}:I\rightarrow \mathbb{R}_{+}$. Suppose $\rho _{n}$ is bounded away from $%
0$ on its support and is of the form

\begin{equation}
\rho _{n}=\psi _{0}+\sum_{k=1}^{n}\frac{\psi _{k}\chi _{k}}{%
|x-x_{k}|^{\alpha _{k}}}\ ,
\end{equation}%
where $\psi_0$ is bounded, $\{x_{k}\}_{k=1,...,n}\subset I$ and, for any $k=1,..,n,$

\begin{itemize}
\item $0<\alpha _{k}<1;$

\item $\psi _{k}$ is continuous at $x_k$ and $\psi _{k}(x_k)\neq 0$;

\item $\chi _{k}$ is either $\mathbf{1}_{[x_{k},+\infty )}$ or $\mathbf{1}%
_{(-\infty ,x_{k}]}.$
\end{itemize}

Then denoting $\alpha :=\underset{k=1,...,n}{\max }\{\alpha _{k}\},$

\begin{equation}
D_{q}(\nu )=%
\begin{cases}
1\ \text{if}\ q<1/\alpha \ , \\ 
\frac{q(1-\alpha )}{q-1}\ \text{otherwise}\ .%
\end{cases}%
\end{equation}
\end{proposition}

\begin{proof}
To ease the computations, we will write the proof assuming that the terms $%
\psi _{k}$ and $\chi _{k}$ are set equal to $1$ for any value of $k.$ This assumption will not affect the result which, as it
clearly appears from the development of the proof, depends only on the
order of the singularities of $\rho _{n}.$ Let us denote $k_{0}$ the
smallest index that realizes $\alpha _{k_{0}}=\alpha $ and set 
\begin{equation}
I_{q}(r):=\int_{I}\nu (B_{r}(x))^{q-1}d\nu (x)\ .
\end{equation}%
All along the proof, $c$ and $c_{i}$ denote positive constants whose values
can change from line to line. Let us first deal with the case $q<0.$ Since 
\begin{equation}
|x-x_{k_{0}}|^{-\alpha }\leq \rho _{n}(x)\leq
c\sum_{k=1}^{n}|x-x_{k}|^{-\alpha }\ ,  \label{q0}
\end{equation}%
we have 
\begin{equation}
I_{q}(r)\leq c\sum_{l=0}^{n}\int_{I}dx(x-x_{l})^{-\alpha }\left[
\int_{x-r}^{x+r}(y-x_{k_{0}})^{-\alpha }dy\right] ^{q-1}\ .  \label{qq0}
\end{equation}%
Since, for all $k=1,...,n,$%
\begin{equation}
\int_{x-r}^{x+r}dy|y-x_{k}|^{-\alpha }\asymp r\mathbf{1}%
_{B_{r}^{c}(x_{k})}(x)+\mathbf{1}_{B_{r}(x_{k})}(x)(|x-x_{k}|+r)^{1-\alpha
}\ ,  \label{qq1}
\end{equation}%
(\ref{qq0}) and (\ref{qq1}) yield 
\begin{align}
I_{q}(r)& \leq cr^{q-1}\sum_{l=1}^{n}\int_{I}dx|x-x_{l}|^{-\alpha }\mathbf{1}%
_{B_{r}^{c}(x_{k_{0}})}(x)+  \label{Iqr} \\
& +c\sum_{1\leq l\leq n\ :\ l\neq k_{0}}\int_{I}dx\left\vert
x-x_{l}\right\vert ^{-\alpha }(|x-x_{k_{0}}|+r)^{(1-\alpha )(q-1)}]\mathbf{1}%
_{B_{r}(x_{k_{0}})}(x)+  \notag \\
& +c\int_{I}dx\left\vert x-x_{k_{0}}\right\vert ^{-\alpha
}(|x-x_{k_{0}}|+r)^{(1-\alpha )(q-1)}\mathbf{1}_{B_{r}(x_{k_{0}})}(x)\ , 
\notag
\end{align}%
The integral in the first term on the r.h.s. of (\ref{Iqr}) is finite, hence
this term is bounded by $c_{1}r^{q-1}.$ The second term on the r.h.s. (\ref%
{Iqr}) is bounded by $c_{2}r^{(1-\alpha )q+\alpha },$ while the third term
is bounded by $c_{3}r^{\left( 1-\alpha \right) q}.$ Therefore,

\begin{equation}
I_{q}(r)\leq c_{1}r^{q-1}+c_{2}r^{(1-\alpha )q+\alpha }+c_{3}r^{\left(
1-\alpha \right) q}\ .
\end{equation}%
Since $q<0$ and $0<\alpha <1,$%
\begin{equation}
I_{q}(r)\leq c_{1}r^{q-1}\left( 1+o(r^{q-1})\right) \ ,
\end{equation}%
then 
\begin{equation}
D_{q}(\nu )\geq 1\ .
\end{equation}%
Since $\mathbb{R}_{+}\ni x\longmapsto x^{q-1}\in \mathbb{R}$ is a convex
function, by (\ref{q0}) and by the Jensen inequality, 
\begin{eqnarray}
I_{q}(r) &\geq &c\int_{I}dx(x-x_{k_{0}})^{-\alpha }\left[ \sum_{k=0}^{n}%
\int_{x-r}^{x+r}(y-x_{k})^{-\alpha }dy\right] ^{q-1}  \label{q1} \\
&\geq &c\left[ \sum_{k=0}^{n}\int_{I}dx(x-x_{k_{0}})^{-\alpha
}\int_{x-r}^{x+r}(y-x_{k})^{-\alpha }dy\right] ^{q-1}\ .  \notag
\end{eqnarray}
Thus, the upper bound of $D_{q}\left( \nu \right) $ is obtained in similar
fashion combining (\ref{q1}) and (\ref{qq1}).

Let us now consider $q>1.$ From (\ref{q0}) we have 
\begin{equation}
I_{q}(r)\geq c\int_{I}dx|x-x_{k_{0}}|^{-\alpha }\left[
\int_{x-r}^{x+r}|y-x_{k_{0}}|^{-\alpha }dy\right] ^{q-1}\ ,
\end{equation}%
which yields, as for the previous cases, 
\begin{equation}
I_{q}(r)\geq c_{1}r^{q-1}+c_{2}r^{(1-\alpha )q+\alpha }+c_{3}r^{\left(
1-\alpha \right) q}\ .
\end{equation}%
If $q>1/\alpha ,$ 
\begin{equation}
I_{q}(r)\geq c_{3}r^{\left( 1-\alpha \right) q}\left( 1+o(r^{\left( 1-\alpha
\right) q})\right) \ ,
\end{equation}%
which implies 
\begin{equation}
D_{q}(\nu )\leq \frac{(1-\alpha )q}{q-1}\ .
\end{equation}%
On the other hand, if $1<q\leq 1/\alpha ,$%
\begin{equation}
I_{q}(r)\geq c_{1}r^{q-1}\left( 1+o(r^{q-1})\right) \ ,
\end{equation}%
which leads to 
\begin{equation}
D_{q}(\nu )\leq 1\ .
\end{equation}%
For what concerns the lower bound, we have 
\begin{eqnarray}
I_{q}(r) &\leq &c\sum_{l=0}^{n}\int_{I}dx|x-x_{l}|^{-\alpha }\left[
\sum_{k=0}^{n}\int_{x-r}^{x+r}|y-x_{k}|^{-\alpha }dy\right] ^{q-1} \\
&\leq &c\sum_{k,l=0}^{n}\int_{I}dx|x-x_{l}|^{-\alpha }\left[
\int_{x-r}^{x+r}|y-x_{k}|^{-\alpha }dy\right] ^{q-1}\ .  \notag
\end{eqnarray}%
Indeed, the last inequality is obtained for $q\geq 2$ by Jensen inequality,
while for $1<q<2,$ it comes from the fact that 
\begin{equation}
(a+b)^{q-1}\leq a^{q-1}+b^{q-1}
\end{equation}%
for any $a,b>0.$ Combining with (\ref{qq1}), we get 
\begin{align}
I_{q}(r)& \leq cr^{q-1}\sum_{k,l=1}^{n}\int_{I}dx|x-x_{l}|^{-\alpha }\mathbf{%
1}_{B_{r}^{c}(x_{k})}(x)+ \\
& +c\sum_{1\leq k\neq l\leq n}\int_{I}dx\left\vert x-x_{l}\right\vert
^{-\alpha }(|x-x_{k}|+r)^{(1-\alpha )(q-1)}\mathbf{1}_{B_{r}(x_{k_{0}})}(x)+
\notag \\
& +c\sum_{1\leq m\leq n}\int_{I}dx\left\vert x-x_{m}\right\vert ^{-\alpha
}(|x-x_{m}|+r)^{(1-\alpha )(q-1)}\mathbf{1}_{B_{r}(x_{m})}(x)\ ,  \notag
\end{align}%
As before, we have the following estimate: 
\begin{equation}
I_{q}^{N}(r)\leq c_{1}r^{q-1}+c_{2}r^{(1-\alpha )q+\alpha }+c_{3}r^{\left(
1-\alpha \right) q}\ .
\end{equation}%
If $q\geq 1/\alpha ,$ we get

\begin{equation}
D_{q}(\nu )\geq \frac{(1-\alpha )q}{q-1}\ ,
\end{equation}%
while if $1<q<1/\alpha ,$

\begin{equation}
D_{q}(\nu )\geq 1\ .
\end{equation}%
Combining all the estimates, we get the thesis for $q\in \mathbb{R}\setminus
\lbrack 0,1]$. Since $D_{q}(\nu )$ is a non-increasing function of $q,$ and since $%
D_{-1}(\nu )=D_{1/\alpha }(\nu )=1,$ we get that $D_{q}(\nu )=1$ for $q\in
\lbrack 0,1],$ which concludes the proof.
\end{proof}\\
We can now finalize the proof of Theorem \protect\ref{propp}.
Let us set $\rho _{0}:=\psi _{0}$ and $\forall k\geq 1$,

\begin{equation}
I\ni x\longmapsto
\rho _{k}\left( x\right) :=\frac{\psi _{k}\left( x\right) \chi _{k}\left(
x\right) }{\left\vert x-x_{k}\right\vert ^{\alpha _{k}}}\in \mathbb{R}_{+}.
\end{equation}
Moreover, let $\nu _{n}:=\sum_{k=0}^{n}\rho _{k}.$ Since $\forall n\geq
1,$

\begin{equation}
1=\nu \left[ 1\right] =\nu _{n}\left[ 1\right] +\int_{I}dx\sum_{k\geq
n+1}\rho _{k}\left( x\right) ,
\end{equation}
setting $c_{n}:=\int_{I}dx\sum_{k\geq
n+1}\rho _{k}\left( x\right) ,$ because $\forall k\geq
1,a_{k}:=\int_{I}dx\rho _{k}\left( x\right)
>0,$ 
\begin{equation}
c_{n}=a_{n+1}+c_{n+1}>c_{n+1} \, .   
\end{equation}
Hence, by the monotone convergence
theorem, 
\begin{equation}
\lim_{n\rightarrow \infty }c_{n}=\inf_{n\geq 0}c_{n}\geq 0 \, .    
\end{equation}
Therefore, 
\begin{eqnarray}
0 &<&\lim_{n\rightarrow \infty }\nu _{n}\left[ 1\right] =\lim_{n\rightarrow
\infty }\int_{I}dx\sum_{k=0}^{n}\rho _{k}\left( x\right) =\lim_{n\rightarrow
\infty }\sum_{k=0}^{n}\int_{I}dx\rho _{k}\left( x\right) = \\
&=&\int_{I}dx\psi _{0}\left( x\right) +\sum_{k\geq 1}\int_{I}dx\rho
_{k}\left( x\right) =1-\inf_{n\geq 0}c_{n}\leq 1\ ,  \notag
\end{eqnarray}%
as well as, $\forall l\geq 1,$%
\begin{equation}
0<\lim_{m\rightarrow \infty }\int_{I}dx\sum_{k=l}^{m}\rho _{k}\left(
x\right) =\lim_{m\rightarrow \infty }\sum_{k=l}^{m}\int_{I}dx\rho _{k}\left(
x\right) =\sum_{k\geq l}\int_{I}dx\rho _{k}\left( x\right) \leq 1\ .
\end{equation}%
Consequently, $$\lim_{n\rightarrow \infty }c_{n}=\lim_{n\rightarrow \infty
}\int_{I}dx\sum_{k\geq n+1}\rho _{k}\left( x\right) =\lim_{n\rightarrow
\infty }\sum_{k\geq n+1}\int_{I}dx\rho _{k}\left( x\right) =0.$$
Thus, $\forall \varepsilon >0,\exists n_{\varepsilon }\geq 1$ such that $%
\forall n>n_{\varepsilon },c_{n}<\varepsilon $ so that 

\begin{equation}
\nu _{n}\left[ 1%
\right] \leq \nu \left[ 1\right] \leq \nu _{n}\left[ 1\right] +\varepsilon .
\end{equation}
Therefore, for any $q>1$ and $n>n_{\varepsilon },$%
\begin{align}
\nu \left[ \left( \nu \left[ \mathbf{1}_{B_{r}\left( \cdot \right) }\right]
\right) ^{q-1}\right] &\leq \nu _{n}\left[ \left( \nu \left[ \mathbf{1}
_{B_{r}\left( \cdot \right) }\right] \right) ^{q-1}\right]
+\int_{I}dx\sum_{k\geq n+1}\rho _{k}\left( x\right) \left( \nu \left[ 
\mathbf{1}_{B_{r}\left( x\right) }\right] \right) ^{q-1}\\
&\leq \nu _{n}\left[ \left( \nu _{n}\left[ \mathbf{1}_{B_{r}\left( \cdot
\right) }\right] +\int_{I}dx\sum_{k\geq n+1}\rho _{k}\left( x\right) \mathbf{%
1}_{B_{r}\left( x\right) }\right) ^{q-1}\right] +c_{n}\\
&\leq \nu _{n}\left[ \left( \nu _{n}\left[ \mathbf{1}_{B_{r}\left( \cdot
\right) }\right] +c_{n}\right) ^{q-1}\right] +c_{n}\\
&\leq \nu _{n}\left[
\left( \nu _{n}\left[ \mathbf{1}_{B_{r}\left( \cdot \right) }\right]
+\varepsilon \right) ^{q-1}\right] +\varepsilon \ .  \notag
\end{align}%
For $q\in \left( 1,2\right) ,$%
\begin{eqnarray}
\nu \left[ \left( \nu \left[ \mathbf{1}_{B_{r}\left( \cdot \right) }\right]
\right) ^{q-1}\right] &\leq &\nu _{n}\left[ \nu _{n}^{q-1}\left[ \mathbf{1}%
_{B_{r}\left( \cdot \right) }\right] +\varepsilon ^{q-1}\right] +\varepsilon
\\
&=&\nu _{n}\left[ \nu _{n}^{q-1}\left[ \mathbf{1}_{B_{r}\left( \cdot \right)
}\right] \right] +\varepsilon ^{q-1}\left( 1+\varepsilon ^{2-q}\right) \ , 
\notag
\end{eqnarray}%
while for $q\geq 2,$%
\begin{align}
\left( \nu _{n}\left[ \mathbf{1}_{B_{r}\left( \cdot \right) }\right]
+\varepsilon \right) ^{q-1} &=\left( \nu _{n}\left[ \mathbf{1}_{B_{r}\left(
\cdot \right) }\right] \right) ^{q-1}+\varepsilon \left( q-1\right)
\int_{0}^{1}ds\left( \nu _{n}\left[ \mathbf{1}_{B_{r}\left( \cdot \right) }%
\right] +s\varepsilon \right) ^{q-2}\\
&\leq \left( \nu _{n}\left[ \mathbf{1}_{B_{r}\left( \cdot \right) }\right]
\right) ^{q-1}+\varepsilon ^{q-1}\ ,  \notag
\end{align}%
that is 
\begin{equation}
\nu _{n}\left[ \left( \nu _{n}\left[ \mathbf{1}_{B_{r}\left( \cdot \right) }%
\right] +\varepsilon \right) ^{q-1}\right] +\varepsilon \leq \nu _{n}\left[
\left( \nu _{n}\left[ \mathbf{1}_{B_{r}\left( \cdot \right) }\right] \right)
^{q-1}\right] +\varepsilon \left( 1+\varepsilon ^{q-2}\right) \ .
\end{equation}%
Setting $\varepsilon _{q}:=\varepsilon ^{q-1}\left( 1+\varepsilon
^{2-q}\right) \wedge \varepsilon \left( 1+\varepsilon ^{q-2}\right) ,\forall
q>1,$%
\begin{eqnarray}
\log \nu \left[ \left( \nu \left[ \mathbf{1}_{B_{r}\left( \cdot \right) }%
\right] \right) ^{q-1}\right] &\leq &\log \left\{ \nu _{n}\left[ \left( \nu
_{n}\left[ \mathbf{1}_{B_{r}\left( \cdot \right) }\right] \right) ^{q-1}%
\right] +\varepsilon _{q}\right\} \\
&=&\log \nu _{n}\left[ \left( \nu _{n}\left[ \mathbf{1}_{B_{r}\left( \cdot
\right) }\right] \right) ^{q-1}\right] +  \notag \\
&&+\varepsilon _{q}\int_{0}^{1}ds\frac{1}{\nu _{n}\left[ \left( \nu _{n}%
\left[ \mathbf{1}_{B_{r}\left( \cdot \right) }\right] \right) ^{q-1}\right]
+s\varepsilon _{q}}\ .  \notag
\end{eqnarray}%
Dividing by $\log r$ and taking the limit as $r\downarrow 0,$ by the
Lagrange mean value theorem, we get $D_{q}\left( \nu \right) \geq
D_{q}\left( \nu _{n}\right) .$
On the other hand, 
\begin{eqnarray}
\nu \left[ \left( \nu \left[ \mathbf{1}_{B_{r}\left( \cdot \right) }\right]
\right) ^{q-1}\right] &=&\left( \nu _{n}\left[ \mathbf{1}_{B_{r}\left( \cdot
\right) }\right] +\int_{I}dx\sum_{k\geq n+1}\rho _{k}\left( x\right) \mathbf{%
1}_{B_{r}\left( x\right) }\right) ^{q-1} \\
&\geq &\left( \nu _{n}\left[ \mathbf{1}_{B_{r}\left( \cdot \right) }\right]
\right) ^{q-1}\ ,  \notag
\end{eqnarray}%
which implies $D_{q}\left( \nu \right) \leq D_{q}\left( \nu _{n}\right) .$\\

Let now $q<1.$ For any $1\leq l\leq n$%
\begin{eqnarray}
\nu \left[ \frac{1}{\left( \nu \left[ \mathbf{1}_{B_{r}\left( \cdot \right) }%
\right] \right) ^{1-q}}\right] &\leq &\nu \left[ \frac{1}{\left( \nu _{n}%
\left[ \mathbf{1}_{B_{r}\left( \cdot \right) }\right] \right) ^{1-q}}\right]
=\nu _{n}\left[ \frac{1}{\left( \nu _{n}\left[ \mathbf{1}_{B_{r}\left( \cdot
\right) }\right] \right) ^{1-q}}\right] + \\
&&+\int_{I}dx\sum_{k\geq n+1}\rho _{k}\left( x\right) \frac{1}{\left( \nu
_{n}\left[ \mathbf{1}_{B_{r}\left( x\right) }\right] \right) ^{1-q}}  \notag
\\
&\leq &\nu _{n}\left[ \frac{1}{\left( \nu _{n}\left[ \mathbf{1}_{B_{r}\left(
\cdot \right) }\right] \right) ^{1-q}}\right] +\int_{I}dx\sum_{k\geq
n+1}\rho _{k}\left( x\right) \frac{1}{\left( \int_{x-r}^{x+r}dy\rho
_{l}\left( y\right) \right) ^{1-q}}\ .  \notag
\end{eqnarray}%
Proceeding as in the proof of the upper bound (\ref{Iqr}) we obtain that
there exist $c_{1},c_{2}>0$ such that 
\begin{eqnarray}
\nu \left[ \frac{1}{\left( \nu \left[ \mathbf{1}_{B_{r}\left( \cdot \right) }%
\right] \right) ^{1-q}}\right] &\leq &c_{1}r^{\left( q-1\right)
}+c_{2}r^{\left( 1-\alpha _{l}\right) q+\alpha _{l}} \\
&=&r^{\left( q-1\right) }\left( c_{1}+c_{2}r^{\alpha _{l}\left( 1-q\right)
+1}\right) \ .  \notag
\end{eqnarray}%
Moreover, since there exists a positive constant $c_{3}$ such that $$%
\int_{I}\nu _{n}\left( dy\right) \mathbf{1}_{B_{r}\left( x\right) }\geq
\int_{x-r}^{x+r}dy\psi _{0}\left( y\right) =c_{3}r,$$%
\begin{eqnarray}
\nu \left[ \frac{1}{\left( \nu \left[ \mathbf{1}_{B_{r}\left( \cdot \right) }%
\right] \right) ^{1-q}}\right] &\geq &\nu _{n}\left[ \frac{1}{\left( \nu _{n}%
\left[ \mathbf{1}_{B_{r}\left( \cdot \right) }\right] +\varepsilon \right)
^{1-q}}\right] \\
&=&\nu _{n}\left[ \frac{1}{\left( \nu _{n}\left[ \mathbf{1}_{B_{r}\left(
\cdot \right) }\right] \right) ^{1-q}}\left( 1-\frac{\varepsilon }{\nu _{n}%
\left[ \mathbf{1}_{B_{r}\left( \cdot \right) }\right] +\varepsilon }\right)
^{1-q}\right]  \notag \\
&\geq &\nu _{n}\left[ \frac{1}{\left( \nu _{n}\left[ \mathbf{1}_{B_{r}\left(
\cdot \right) }\right] \right) ^{1-q}}\right] \left( 1-\frac{\varepsilon }{%
c_{3}r+\varepsilon }\right) ^{1-q}\ .  \notag
\end{eqnarray}%
Hence, for $q<1$ we get $1\leq D_{q}\left( \nu \right) \leq D_{q}\left( \nu
_{n}\right).$ Since by Proposition \ref{propp0}, for $q<1,D_{q}\left( \nu
_{n}\right) =1,$ this implies $D_{q}\left( \nu \right) =D_{q}\left( \nu
_{n}\right)=1 .$ Finally, since $D_q(\nu)$ is a non-increasing function of $q$, we get 
$D_{1}\left( \nu \right) =1.$\\

\section{Aknowledgements}

TC was partially supported by CMUP, which is financed by national funds
through FCT -- Funda\c{c}ao para a Ci\^{e}ncia e Tecnologia, I.P., under the
project with reference UIDB/00144/2020. TC also acknowledges the Department
of Mathematics and Computer Science of the University of Calabria, where part of
this research was carried on, for hospitality and support under the VIS 2023
project. TC thanks J\'{e}r\^{o}me Rousseau and Mike Todd for useful
discussions concerning generalized dimensions and Jorge M. Freitas for his
expertise on unimodal maps. MG is partially supported by GNAMPA and
acknowledges the CMUP, where part of this work was done, for hospitality and
support. Both authors thank Sandro Vaienti for his insightful comments.

\end{document}